\documentclass[pdflatex,sn-mathphys-num]{sn-jnl}

\usepackage{graphicx}%
\usepackage{multirow}%
\usepackage{amsmath,amssymb,amsfonts}%
\usepackage{amsthm}%
\usepackage{mathrsfs}%
\usepackage[title]{appendix}%
\usepackage{xcolor}%
\usepackage{textcomp}%
\usepackage{manyfoot}%
\usepackage{booktabs}%
\usepackage{algorithm}%
\usepackage{algorithmicx}%
\usepackage{algpseudocode}%
\usepackage{listings}%

\theoremstyle{thmstyleone}%
\newtheorem{theorem}{Theorem}%
\newtheorem{proposition}[theorem]{Proposition}%
\newtheorem{lemma}[theorem]{Lemma}

\theoremstyle{thmstyletwo}%
\newtheorem{remark}{Remark}%

\theoremstyle{thmstylethree}%
\newtheorem{definition}{Definition}%

\DeclareMathOperator\sinc{sinc}

\usepackage{enumitem}%

\begin{document}

\title[Convergence of the Waveholtz Iteration on $\mathbb{R}^d$]{Convergence of the Waveholtz Iteration on $\mathbb{R}^d$}

\author*[1]{\fnm{Olof} \sur{Runborg}}\email{olofr@kth.se}

\author[2]{\fnm{Elliot} \sur{Backman}}\email{elliotba@kth.se}

\affil*[1]{\orgdiv{Department of Mathematics}, \orgname{KTH Royal Institute of Technology}, \orgaddress{\street{Lindstedtsvägen 25}, \city{Stockholm}, \postcode{11428}, \country{Sweden}}}

\affil[2]{\orgdiv{Department of Mathematics}, \orgname{KTH Royal Institute of Technology}, \orgaddress{\street{Lindstedtsvägen 25}, \city{Stockholm}, \postcode{11428}, \country{Sweden}}}

\abstract{In this paper we analyse the Waveholtz method, a time-domain iterative method for solving the Helmholtz iteration, in the constant-coefficient case in all of $\mathbb{R}^d$. We show that the difference between a Waveholtz iterate and the outgoing Helmholtz solution satisfies a Helmholtz equation with a particular kind of forcing. 
For this forcing, 
we prove a frequency-explicit estimate in weighted Sobolev norms, that shows a decrease of the differences as $1/\sqrt{n}$ in terms of the iteration number $n$. This guarantees the convergence of the real parts of the Waveholtz iterates to the real part of the outgoing solution of the Helmholtz equation. 
}



\maketitle

\section{Introduction}\label{sec1}

The Helmholtz equation 
\begin{equation} \label{eq:variableHelmholtz}
\nabla\cdot(c^2\nabla u) + \omega^2u = f,\qquad \text{in $\Omega$,}
\end{equation}
subject to appropriate boundary conditions is a useful model for a variety of real-world wave phenomena, such as the propagation of electromagnetic, underwater, or seismic waves. In this paper we consider an iterative
time-domain method called Waveholtz \cite{Waveholtz2020}
for numerically solving the Helmholtz equation. In particular, we present a proof of convergence when the domain $\Omega$ is the full space $\mathbb{R}^d$. Numerical approximation of solutions to the Helmholtz equation presents several difficulties, especially when the frequency parameter $\omega$ is large. Discretisation of such equations produces large linear systems of equations, and solving these systems by use of Krylov space methods is relatively difficult, as the matrices associated to the systems are indeterminate and large. Furthermore standard preconditioners do not work well; see for instance \cite{ErnstGander2012}. These issues have motivated the study of specialised preconditioners of the linear systems, including among others shifted Laplacian preconditioners \cite{EOVshiftedLaplacian},
the analytic incomplete LU preconditioner \cite{GanderNataf}, sweeping preconditioners \cite{EngYingSweeping} as well as methods based on domain decomposition \cite{Stolk,ChenXiang, VionGeuzaine}. Despite the progress that has been made, difficulties remain especially in the high-frequency context, for interior problems, and for problems with high-contrast material parameters. Efficient implementation of high-order versions of the preconditioned methods on large computers is also a challenge.
The Waveholtz method, introduced in \cite{Waveholtz2020}, aims to alleviate some of these issues by transferring the Helmholtz equation into a time-domain setting and working with an associated wave equation. Potential benefits of this approach include the existence of memory-lean, parallelizable, and high-order provably stable numerical methods for solving the wave equation. The time-domain approach 
for solving Helmholtz has also been used
in the Controllability Method \cite{Controllability}
and for preconditioning \cite{Stolk2}.

The Waveholtz method is an iterative method aimed at computing a solution to the Helmholtz equation by finding the fixed point of an affine operator, the evaluation of which requires solving an associated wave equation. Indeed, the method is motivated by the observation that a solution $u$ to the Helmholtz equation formally defines a fixed point of the map $\Pi$ given by 
\begin{equation} \label{eq:PiDef}
\Pi v(x) := \int_0^T K(t)
    w(x, t)  dt,    
\end{equation}
where $T = 2\pi/\omega$ is the period  corresponding to the frequency parameter $\omega$ and the function $w(x, t)$ is the solution to the associated wave equation
\begin{align}
\label{eq:variableWaves}
    \partial_t^2w  &= \nabla \cdot(c^2(x)\nabla w) - f(x)\cos(\omega t),& \qquad (x, t) &\in \Omega \times (0, T), \nonumber\\
    w(x, 0) &= v(x),& x&\in \Omega,
    \\
    \partial_t w(x, 0) &= 0,& x&\in \Omega,
    \nonumber
\end{align}
with boundary conditions consistent with those of \eqref{eq:variableHelmholtz}. The function $K(t)$ is a kernel given by 
\begin{equation} \label{eq:kernel}
K(t) = \frac{2}{T}\Big(\cos(\omega t) - \frac{1}{4}\Big).    
\end{equation}
A function $u$ that solves $(\ref{eq:variableHelmholtz})$ 
then formally defines a fixed point of $\Pi$, since the function $v(x, t) = u(x)\cos(\omega t)$ formally solves the initial-value problem $(\ref{eq:variableWaves})$ for $v(x) = u(x)$ so that one would have
\begin{equation} \label{eq:iterationDef}
\Pi u = u(x)\int_0^T K(t)\cos(\omega t)dt = u.   
\end{equation}
In fact, the choice of $K(t)$ is made precisely so that the integral over $(0, T)$ of $\cos(\omega t)$ against $K(t)$ will equal $1$. The Waveholtz method then consists of iterating the map $\Pi$ from the starting point $u^0 \equiv 0$ to find a sequence of functions 
$$
u^n := \Pi^nu^0
$$ 
which converges towards the fixed point $u$ that solves the Helmholtz equation $(\ref{eq:variableHelmholtz})$.

In this paper we will analyze the convergence of the Waveholtz method.  Earlier work on convergence has considered bounded domains. In \cite{Waveholtz2020} it is shown for homogeneous Dirichlet or Neumann boundary conditions and variable wave speed $c(x)$ that the iterates converge in $H^1$-norm with convergence rate $1-O(\delta^2)$ to the solution of the Helmholtz equation, where $\delta$ is the relative gap between the Helmholtz frequency $\omega$ and the nearest eigenvalue of the operator $-\nabla\cdot(c^2\nabla)$. Additional convergence results for the Dirichlet and Neumann case were given in \cite{AGR2022}. The proofs of these results have relied on the possibility to decompose functions of interest into a sum of eigenfunctions of the operator $-\nabla\cdot(c^2\nabla)$, a strategy which is naturally only available for domains and boundary conditions where the existence of an orthonormal basis of eigenfunctions is guaranteed. This is the case for Neumann or Dirichlet boundary conditions. 
However, such systems of eigenfunctions do not exist for instance in the case of bounded domains with impedance or absorbing conditions. In these cases there are no theoretical convergence results, except in one dimension \cite{AGR2022}. Nonetheless, there is ample numerical evidence of the convergence of the Waveholtz method in more general settings; convergence has been studied numerically for a variety of geometries and boundary conditions \cite{WaveholtzABHS2025, Waveholtz2020, WaveholtzEL, WaveholtzEM}. The observed convergence however remains to be fully explained theoretically.

In this paper we study convergence of Waveholtz for the constant-coefficient Helmholtz equation set in the full space ${\mathbb R}^d$,
\begin{equation} \label{eq:HE}
\Delta u + \omega^2 u = f(x), \qquad x\in \mathbb{R}^d, 
\end{equation}
subject to the Sommerfeld condition
\[
\lim_{|x| \rightarrow \infty} |x|^{\frac{d-1}{2}}\big(\partial_r - i\omega\big)u(x) = 0.
\] 
This model problem shares some similarities to the case of bounded domains
with impedance or absorbing boundary conditions. Crucially neither of these two kinds of problems allows for decomposing functions in $L^2$ into eigenfunctions of the operator $-\nabla \cdot(c^2(x)\nabla)$. The associated wave equation is in this case 
\begin{align}
\label{eq:Waves}
    \partial_t^2w &= \Delta w-f(x)\cos(\omega t),& \qquad (x, t) &\in \mathbb{R}^d \times (0, T), \nonumber\\
    w(x, 0) &= v(x),& x&\in \mathbb{R}^d,
    \\
    \partial_t w(x, 0) &= 0,& x&\in \mathbb{R}^d.
    \nonumber
\end{align}
The main result of this paper is Theorem \ref{thm:convergence} which shows a frequency explicit estimate of the difference between real parts of the Waveholtz iterates and the outgoing solution of \eqref{eq:HE}. In particular, these differences decreases as $n^{-\frac{1}{2}}$ where $n$ is the iteration number, which implies that the Waveholtz method applied
to \eqref{eq:HE} produces iterates whose real parts converge to the real part of the outgoing solution of \eqref{eq:HE}. Note that solutions to \eqref{eq:variableHelmholtz} and \eqref{eq:HE} are not in $L^2(\mathbb{R}^d)$. The results in Theorem \ref{thm:convergence} are therefore given in norms for weighted Sobolev spaces. This introduces some technical difficulties as the domain of the Waveholtz operator $\Pi$ needs to be extended to these spaces in the analysis.

The paper is organized as follows. First in Section \ref{sec:prelim} we introduce some notation, terminology, and results that are used in the proofs that follow. Second, in Section \ref{sec:convergence} we state and prove our convergence result for the Waveholtz iteration. The proof in Section \ref{sec:convergence} makes use of an estimate for solutions of particular kinds of Helmholtz equations. The statement and proof of this estimate is found in Section \ref{sec:HelmholtzEstimate}. Finally, the proof of a technical result used in the proof of Theorem \ref{thm:convergence} is found in Appendix \ref{sec:A1}.

\section{Preliminaries}\label{sec:prelim}

It is well-known that the Helmholtz equation on $\mathbb{R}^d$ in general does not have solutions in the usual Sobolev space $H^1(\mathbb{R}^d)$, so that in order to analyse the Helmholtz equation on an unbounded domain it is required to define appropriate weighted function spaces. The approach that we will take is that of Agmon in \cite{Agmon}, in which the following spaces are defined:
\begin{definition}
    \label{weightedLp}
    The \emph{weighted $L^p$ space} $L^p_s(\mathbb
    {R}^d)$
    is defined by 
    \[
    L^p_s(\mathbb
    {R}^d) = \Big\{ f  : \big\|f \langle x\rangle^s\big\|_{L^p(\mathbb{R}^d)} < \infty \Big\},
    \]
    where $\langle x \rangle := (1 + |x|^2)^{-1/2}$. For $1\leq p\leq \infty$ the vector space is a Banach space with norm 
    \[
    \|f\|_{L^p_s(\mathbb{R}^d)} := \|f \langle x\rangle^s\|_{L^p(\mathbb{R}^d)}.
    \]
    The \emph{weighted Sobolev spaces} $H^{k}_s(\mathbb{R}^d)$ are defined similarly by 
    \[
    H^{k}_s(\mathbb{R}^d) = \Big\{ f : \big\|f \big\|_{H^{k}_s(\mathbb{R}^d)} < \infty \Big\},
    \]
    with norm 
    \[
    \|f\|_{H^k_s(\mathbb{R}^d)} = \bigg(\sum_{|\alpha|\leq k} \|\partial^\alpha f\|^2_{L^2_s(\mathbb{R}^d)}\bigg)^{\frac{1}{2}}.
    \]
\end{definition}
We will denote by $\hat{f}$ the Fourier transform of $f\in L^2(\mathbb{R}^d)$, using the convention
\[
\hat{f}(\xi) = \int_{\mathbb{R}^d} f(x) \exp(-2\pi i \xi x)dx.
\]
We also denote by $C^\infty_0(\Omega)$ the set of compactly supported smooth functions on the open set $\Omega \subset \mathbb{R}^d$. The reader may recall that there are several equivalent norms one might use to define the spaces $H^k(\mathbb{R}^d)$. We will occasionally use the norm given by the relationship 
\begin{equation} \label{eq:SobolevWeightedLpEquiv}
    \|f\|_{H^k(\mathbb{R}^d)} = \|\langle\xi\rangle^k\hat{f}(\xi)
\|_{L^2(\mathbb{R}^d)},
\end{equation}
for $k \in \mathbb{R}$. Using this convention it becomes clear from Definition \ref{weightedLp} that $f \in H^s(\mathbb{R}^d)$ precisely when $\hat{f} \in L^2_s(\mathbb{R}^d)$. We also note that for $s_2\leq s_1,$ we have
    \[
    L^2_{s_1}(\mathbb{R}^d) \subset L^2_{s_2}(\mathbb{R}^d),
    \]
where the embeddings are continuous. This provides some intuition about the spaces: the larger the value $s$ is, the more restrictive the space $L^2_{s}(\mathbb{R}^d)$ becomes, as in order for the norm $\|\cdot\|_{L^2_{s}(\mathbb{R}^d)}$ to be finite a function in $L^2_{s}(\mathbb{R}^d)$ must decay rapidly enough to compensate for the weight $\langle x \rangle^s$, which grows on the order of $|x|^s$. The case $s = 0$ produces the standard $L^2$-space, whereas a negative weight $-s$ allows for a certain growth of the functions $f\in L^2_{-s}(\mathbb{R}^d)$, which can be controlled by the weight $\langle x \rangle^{-s}$. We shall also make use of the following trace estimate.
\begin{proposition}[Theorem 9.4 in \cite{LionsMagenes}]\label{prop:traceEst}
    Suppose $\Omega \subset \mathbb{R}^d$ is an open bounded subset such that its boundary $\partial \Omega$ is a $(d-1)$-dimensional smooth manifold. Then if $s > \frac{1}{2}$ the \emph{trace operator} $\tau: H^s(\Omega) \longrightarrow H^{s-\frac{1}{2}}(\partial\Omega)$ is bounded, satisfying
    \begin{equation*} 
        \|\tau v \|_{H^{s-\frac{1}{2}}(\partial \Omega)} \leq K\| v \|_{H^{s}(\Omega)},
    \end{equation*}
    for some $K \in \mathbb{R}^+$ and every $v\in H^s(\Omega)$.
\end{proposition}

The following Proposition demonstrates the usefulness of the spaces $H^k_s(\mathbb{R}^d)$ when studying the Helmholtz equation.

\begin{proposition}[Theorem 4.1 (i) in \cite{Agmon}] \label{prop:LAP}
Suppose $s > \frac{1}{2}$, and let 
\[
R(z) := (-\Delta-z)^{-1} : L^2_s(\mathbb{R}^d) \longrightarrow H^2_{-s}(\mathbb{R}^d),
\]
be the resolvent operator of $-\Delta$. For any $\omega^2 \in \mathbb{R}^+$ the limit 
\[
\underset{\mathfrak{Im}\{z\} > 0}{\lim_{z \rightarrow \omega^2}} R(z) =: R(\omega^2),
\]
exists as an element of the space of bounded linear operators between the spaces $L^{2}_s(\mathbb{R}^d)$ and $H^2_{-s}(\mathbb{R}^d)$, endowed with the uniform operator topology.
\end{proposition}
The proposition shows that any source term $f \in L^2_s(\mathbb{R}^d)$ in the Helmholtz equation on $\mathbb{R}^d$ produces a  solution $u = R(\omega^2)f \in H^2_{-s}(\mathbb{R}^d)$. Solutions $u$ defined in this way will be called \textit{outgoing}. These solutions are precisely those satisfying the Sommerfeld condition 
\[
\lim_{|x| \rightarrow \infty} |x|^{\frac{d-1}{2}}\big(\partial_r - i\omega\big)u(x) = 0.
\]
That the limit $R(z) \longrightarrow R(\omega^2)$ in the proposition exists in the uniform operator topology means that
\[
\|R(z) - R(\omega^2)\| \longrightarrow 0,
\]
as $z \longrightarrow \omega^2$ with $z$ such that $\mathfrak{Im}\{z\} > 0$, where $\| \cdot \|$ here denotes the standard norm on the space of bounded linear operators from $L^{2}_s(\mathbb{R}^d)$ into $H^2_{-s}(\mathbb{R}^d)$. This is a version of the so-called \emph{limiting absorption principle}: the operator $-\Delta-\omega^2$ has a well-defined resolvent which can be understood as the limit as $\alpha \rightarrow 0^+$ of resolvents of the type $\big(-\Delta-(\omega^2 - i\alpha)\big)^{-1}$. 

\section{Convergence of Waveholtz} \label{sec:convergence}

In the section we will prove the convergence of Waveholtz iterates in the open case \eqref{eq:HE} where the Helmholtz solution  is sought in the whole space $\mathbb{R}^d$. The main result is Theorem~\ref{thm:convergence} which gives a frequency-explicit estimate of the real part of the error in iteration $n$. In what follows we denote by $f$ the forcing used in \eqref{eq:HE} and by $u$ the corresponding outgoing solution. We denote by $u^n$ the the Waveholtz iterates defined by
\[
    u^{n+1} = \Pi u^n, \qquad u^0(x) \equiv 0,
\]
with $\Pi$ as in \eqref{eq:iterationDef}. Each iterate produces a corresponding error $e^n$ defined by $e^n = u^n - u$. With this we are ready to state our convergence result. 
\begin{theorem}
\label{thm:convergence}
    Suppose $s > \frac{3}{2}$ and $f \in L^2_s(\mathbb{R}^d)$. Then for $\omega \geq 1$ the Waveholtz iterates $u^n$ converge in $H^1_{-s}(\mathbb{R}^d)$-norm to the real part of the outgoing solution $u$ to (\ref{eq:HE}), with the error $e^n$ at the $n$th iteration satisfying
    \[
    \|\mathfrak{Re}\{e^n\}\|_{L^2_{-s}(\mathbb{R}^d)} \leq C\omega^{2s-2}n^{-\frac{1}{2}}\|f\|_{L^2_s(\mathbb{R}^d)},
    \]
    where $C$ is independent of $f$, $u$, $\omega$ and $n$.
    If also $f \in H^1_s(\mathbb{R}^d)$ we have the estimate
    \[
    \|\mathfrak{Re}\{e^n\}\|_{H^1_{-s}(\mathbb{R}^d)} \leq C\omega^{2s-1}n^{-\frac{1}{2}}\|f\|_{H^1_s(\mathbb{R}^d)},
    \]
    for the same $C$.
\end{theorem}

\begin{remark}
The frequency-explicit result in Theorem~\ref{thm:convergence} can be used to estimate how the convergence rate depends on $\omega$. We note first that the solution $u$ also depends on $\omega$ so to be able to compare errors for different frequencies we will choose the source $f$ of the form
$$
 f(x) = \omega^{\frac{3+d}{2}}S(\omega x),
$$
with $S\in H^1_s({\mathbb R}^n)$ for some $s>3/2$. In the high-frequency regime $\omega\to\infty$, this source concentrates in the origin and generates a solution of size $O(1)$ in $\omega$ for constant coefficients; see for instance scaling discussions in \cite{CasPerRun:02}. Following the same derivation as in Section~\ref{sec:hfest} we get $g=\omega^{\frac{d-1}{2}}S(x)$ and the estimate \eqref{eq:gest} for $g$ is then replaced by
    \begin{equation*}
    \|g\|_{L^2_{s}(\mathbb{R}^d)} = \omega^{\frac{d-1}{2}}\|S\|_{L^2_s(\mathbb{R}^d)}. 
    \end{equation*}
The final estimate becomes
    \[
    \|\mathfrak{Re}\{e^n\}\|_{L^2_{-s}(\mathbb{R}^d)} \leq 
C\omega^{s-\frac12}
{\mathcal N}_\omega(\beta^n)
\|S\|_{L^2_{s}(\mathbb{R}^d)}
\leq 
C\omega^{s-\frac12}n^{-\frac{1}{2}}\|S\|_{L^2_s(\mathbb{R}^d)}.
    \]
Since the solution is $O(1)$ in $\omega$ this error corresponds to the relative error. It thus scales as $\omega^{s-1/2}n^{-1/2}$, and to reduce it below a fixed given tolerance when $\omega$ grows, $n$ needs to scale as $n\sim \omega^{2s-1}$. Since $s$ must be larger than $3/2$ in Theorem~\ref{thm:convergence}, the best scaling that can be proved is $n\sim \omega^{2^+}$. However, in numerical computations on finite domains with absorbing type boundary conditions, the scaling $n\sim \omega$ is typically observed. We therefore conjecture that the optimal limit for $s$ in Theorem~\ref{thm:convergence} is actually $s>1$, as this would give the scaling $n\sim \omega^{1^+}$.
\end{remark}

\begin{remark} 
To further understand the error we suppose that the source $S$ has compact support in $|x|\leq 2\pi $, i.e. that $f$ has compact support in $|x|\leq T=2\pi/\omega$. We then note that because of finite speed of propagation, the Waveholtz approximation $u^n(x)$ will be identically zero outside the the ball $|x|\leq (n+1)T$. The total error in $u^n$ will therefore be composed of an $O(1)$ exterior error, outside that ball, and an interior error, inside the ball. Assuming $|u|\sim 1/{|x|^{(d-1)/2}}$, the exterior error will be
$$
  ||u_n - u||_{L^2_{-s}(|x|>(n+1)T)}
  \sim \left(\int_{(n+1)T}^\infty
  \frac{r^{d-1} r^{-2s}dr}{r^{d-1}}\right)^{\frac12}
  = \frac{((n+1)T)^{-s+\frac12}}
  {\sqrt{2s-1}}
  \sim \omega^{s-\frac12}n^{-s+\frac12}.
$$
The exterior error provides a lower bound for the total error. Comparing it to the upper bound of the total error derived in the previous remark, we note that the two errors scale in the same way with $\omega$ and $n$ when $s=1$. The scaling of the error estimate would thus be optimal in this case, but again require that $s>1$ was the limit in Theorem~\ref{thm:convergence}.
\end{remark}

\subsection{Proof of Theorem 3.1}
The proof of Theorem 3.1 is given in two steps. First we show that the error $e^n$ is the outgoing solution to a Helmholtz equation
\begin{equation} \label{eq:ErrorEquation}
\Delta e^n + \omega^2 e^n = F_n,  
\end{equation}
with a source term $F_n$ of a certain form which will be given shortly.  Second, we apply an estimate derived in Section \ref{sec:HelmholtzEstimate} for the Helmholtz equation with this type of source term, which allows us to estimate $e^n$.

\subsubsection{The error equation}
To derive \eqref{eq:ErrorEquation} we show a sequence of short lemmata.
\begin{lemma} \label{le:regularity}
    Suppose $g \in L^2(\mathbb{R}^d)$ and $\alpha > 0$. Then a solution to 
    \[
    \Delta v + \omega^2 v +i\omega\alpha v = g,
    \]
    satisfies $v \in H^2(\mathbb{R}^d)$.
\end{lemma}

\begin{proof}
    If we apply the Fourier transform to the equation we find
    \[
    \hat{v}(\xi) = \frac{\hat{g}(\xi)}{-|\xi|^2 + \omega^2 + i\alpha\omega}.
    \]
    By the Plancherel theorem
    \begin{align*}
        \|v\|_{H^2(\mathbb{R}^d)}^2 &= \int_{\mathbb{R}^d}(1 + |\xi|^2)^2|\hat{v}|^2d\xi = \int_{\mathbb{R}^d}\frac{(1+|\xi|^2)^2}{(\omega^2 - |\xi|^2)^2 + \alpha^2\omega^2} |\hat{g}|^2 d\xi 
        \\
        &\leq C(\alpha, \omega)\int_{\mathbb{R}^d}|\hat{g}|^2d\xi =C(\alpha, \omega)\|g\|_{L^2(\mathbb{R}^d)}^2.
    \end{align*}
    Since $\|g\|_{L^2(\mathbb{R}^d)} < \infty$ by assumption, we conclude that $v \in H^2(\mathbb{R}^d)$.
\end{proof}

Key to the analysis of the Waveholtz iteration is the operator $\mathcal{S}$ defined on $C^\infty_0(\mathbb{R}^d)$ by
\begin{equation}\label{eq:Sdef}
\mathcal{S}v(x) =\int_0^T K(t)w(x, t)dt,
\end{equation}
where $w(x, t)$ solves the initial-value problem
\begin{align}
\label{eq:systemS}
    \partial_t^2w &= \Delta w,& \qquad (x, t) &\in \mathbb{R}^d \times (0, T), \nonumber\\
    w(x, 0) &= v(x),& x&\in \mathbb{R}^d,
    \\
    \partial_t w(x, 0) &= 0,& x&\in \mathbb{R}^d.
    \nonumber
\end{align}
We note that the operator $\mathcal{S}$ is a special case of the operator $\Pi$, corresponding to a source term $f = 0$ on the domain $\mathbb{R}^d$. In the analysis of the Waveholtz iteration on an open bounded domain $\Omega$ it is sufficient to find a continuous extension of this operator into $L^2(\Omega)$ or $H^1(\Omega)$, which is relatively simple to do. In the unbounded case, however, we will need to extend the operator $\mathcal{S}$ into the weighted spaces in which we expect to find our solutions to the Helmholtz equation. The fact that this is possible is demonstrated in the next proposition.

\begin{proposition} \label{prop:extensionS}
The operator $\mathcal{S}$ extends uniquely to a bounded linear operator
from $L^2_s(\mathbb{R}^d)$ to $L^2_s(\mathbb{R}^d)$ and from $H^p_s(\mathbb{R}^d)$ to $H^p_s(\mathbb{R}^d)$ for all $s \in \mathbb{R}$ and postive integers $p$. Furthermore, functions $v$ in $L^2(\mathbb{R}^d)$ satisfy 
    \begin{equation} \label{eq:SFourier0}
    \widehat{\mathcal{S}v}(\xi) = \beta(|\xi|)\hat{v},(\xi)    
    \end{equation}
where 
    \begin{equation} \label{eq:betaDef}
    \beta(\lambda) = \int_0^TK(t)\cos(\lambda t)dt.    
    \end{equation}
Provided the source term in \eqref{eq:HE}
satisfies
$f \in L^2_t(\mathbb{R}^d)$ and $t > \frac{1}{2}$, the operator $\Pi$ also extends as an operator on $H^p_s(\mathbb{R}^d)$, and
\begin{equation} \label{eq:PiSrelationship}
\Pi v = \mathcal{S}(v-u) + u,    
\end{equation}
for the outgoing solution $u$ to \eqref{eq:HE}.
\end{proposition}
The proof of this result is given in the Appendix. We also note that the property \eqref{eq:PiSrelationship} implies that $u$ is a fixed point of $\Pi$, since
\[
\Pi u = \mathcal{S}(u - u ) + u = \mathcal{S}0 + u = u,
\]
as $\mathcal{S}$ is a linear operator. Furthermore, since $e^{n+1} = \Pi u^n - u$ by definition, we find
\begin{equation} \label{eq:errorS}
    \mathcal{S}e^n = \mathcal{S}(u^n - u) = \Pi u^n - u = e^{n+1}.
\end{equation}
We now show that the operator $\mathcal{S}$ commutes with the Laplace operator for sufficiently smooth functions.
\begin{lemma} \label{le:commutativity}
    If $v \in H^2(\mathbb{R}^d)$, then
    \[
    \big( \Delta \circ \mathcal{S}\big) v = \big(\mathcal{S} \circ \Delta\big) v
    \]
\end{lemma}
\begin{proof}
    By Proposition \ref{prop:extensionS} we know that for any $v \in H^2(\mathbb{R}^d)$ we have $\mathcal{S}v \in H^2(\mathbb{R}^d)$ as well. This means that $\Delta\mathcal{S}v$ is well-defined and lies in $L^2(\mathbb{R}^d)$. Similarly, since $v \in H^2(\mathbb{R}^d)$ we have $\Delta v \in L^2(\mathbb{R}^d)$ so that by Proposition \ref{prop:extensionS} we have $\mathcal{S}(\Delta v) \in L^2(\mathbb{R}^d)$. We want to show that $\mathcal{S}(\Delta v) = \Delta(\mathcal{S}v)$ as elements of $L^2(\mathbb{R}^d)$. As the Fourier transform is an isometry of the space $L^2(\mathbb{R}^d)$ onto itself, this equality holds precisely if the corresponding equality in the Fourier domain,
    \[
    \widehat{\big(\mathcal{S}(\Delta v)\big)}(\xi) = \widehat{\big(\Delta(\mathcal{S}v)\big)}(\xi),
    \]
    holds. Let us investigate this. On the one hand Proposition \ref{prop:extensionS} tells us that
    \[  
    \widehat{\big(\Delta(\mathcal{S}v)\big)}(\xi) = -|\xi|^2 \widehat{\mathcal{S}v}(\xi) = -|\xi|^2\beta(|\xi|)\hat{v}(\xi),
    \]
    whilst on the other we have 
    \[
    \widehat{\big(\mathcal{S}(\Delta v)\big)}(\xi) = \beta(|\xi|)\widehat{\Delta v}(\xi) = -|\xi|^2\beta(|\xi|)\hat{v}(\xi).
    \]
    We see that the two functions have the same image under the Fourier transform, implying that they must be equal. This concludes the proof.
\end{proof}

\begin{lemma} \label{le:outgoing}
    Suppose $s > \frac{1}{2}$, $g \in L^2_s(\mathbb{R}^d)$. Denote by $v$ and $w$ the outgoing solutions to 
    \[
    \Delta v + \omega^2 v = g,
    \]
    and
    \[
    \Delta w + \omega^2 w = \mathcal{S}g.
    \]
    These exist and lie in $H^2_{-s}(\mathbb{R}^d)$. Moreover, $w = \mathcal{S}v$. 
\end{lemma}
\begin{proof}
    Let us denote by $v_\alpha$ the solution to
    \[
    \Delta v_\alpha + \omega^2 v_\alpha +i\omega\alpha v_\alpha= g.
    \]
    Lemma \ref{le:regularity} tells us that these $v_\alpha$ lie in $H^2(\mathbb{R}^d)$, so that we by Lemma $\ref{le:commutativity}$ find
    \[
    \mathcal{S}\big(\Delta v_\alpha + \omega^2 v_\alpha + i\omega\alpha v_\alpha\big) = \Delta \mathcal{S}v_\alpha + \omega^2 \mathcal{S}v_\alpha +i\omega\alpha \mathcal{S}v_\alpha= \mathcal{S}g. 
    \]
    That is, the function $w_\alpha = \mathcal{S}v_\alpha$ satisfies
    \[
    \Delta w_\alpha + \omega^2 w_\alpha + i\omega\alpha w_\alpha= \mathcal{S}g.
    \]
    Since $g$ and $\mathcal{S}g$ both lie in $L^2_s(\mathbb{R}^d)$ by Proposition \ref{prop:extensionS}, Proposition \ref{prop:LAP} ensures that  
    \[
    v = \lim_{\alpha \rightarrow 0^+}v_\alpha \qquad \text{and} \qquad w = \lim_{\alpha \to 0^+} \mathcal{S}v_\alpha,
    \]
    exist as functions in $H^2_{-s}(\mathbb{R}^d)$. Furthermore, Proposition \ref{prop:extensionS} tells us that $\mathcal{S}: H^2_{-s}(\mathbb{R}^d) \to H^2_{-s}(\mathbb{R}^d)$ is bounded, which means that
    \[
    w = \lim_{\alpha \rightarrow 0^+} \mathcal{S}v_\alpha = \mathcal{S}\big(\lim_{\alpha \rightarrow 0^+} v_\alpha \big) = \mathcal{S}v.
    \]
    This concludes the proof.
\end{proof}

\begin{proposition} \label{prop:errorEq}
Suppose $f \in L^2_{s}(\mathbb{R}^d)$ with
    $s > \frac{1}{2}$. Then the error $e^n$ of the $n$th Waveholtz iterate is an outgoing solution to \eqref{eq:ErrorEquation}, with  $F_n = -\mathcal{S}^nf$, that is
    \begin{equation} \label{eq:errorEq}
    \Delta e^n + \omega^2 e^n = -\mathcal{S}^nf.
    \end{equation}
    
\end{proposition}

\begin{proof}
    By definition $e^0 = u_0-u=-u$, which is the outgoing solution to 
    \[
    \Delta e^0 + \omega^2 e^0 = -f.
    \]
    Since $f \in L^2_{s}(\mathbb{R}^d)$ and $\mathcal{S}$ is bounded on this domain by Proposition \ref{prop:extensionS}, we get $-\mathcal{S}^n f \in L^2_{s}(\mathbb{R}^d)$
    for all $n\geq 0$. Induction
    then gives the result upon observing that if
    $e^n$ is an outgoing solution to
    $\Delta e^n + \omega^2 e^n = -\mathcal{S}^nf$,
    then
    Lemma \ref{le:outgoing} 
    implies that $e^{n+1}$ is the outgoing solution to $\Delta e^{n+1} + \omega^2 e^{n+1} = -\mathcal{S}^{n+1}f$, since
    $e^{n+1}=\mathcal{S}e^n$ by \eqref{eq:errorS}.
\end{proof}

\subsubsection{Estimate of the error}

Having shown that the error $e^n$ itself is the outgoing solution to the Helmholtz equation \eqref{eq:ErrorEquation} we can study properties of such solutions to understand the behaviour of $e^n$ as $n \rightarrow \infty$. We know by Proposition \ref{prop:errorEq} that $e^n$ satisfies $(\ref{eq:errorEq})$ and by Proposition \ref{prop:extensionS} that 
\begin{equation} \label{eq:SnBeta}
    \widehat{\mathcal{S}^nf}(\xi) = \beta^n(|\xi|)\hat{f}(\xi).
\end{equation}
If $f$, $\beta^n$, and $s$ were to satisfy the relevant assumptions (H1) through (H3) in Section~\ref{sec:HelmholtzEstimate}, Theorem \ref{thm:HEBound} would therefore give us the bound 
\begin{equation} \label{eq:errorBound}
    \|\mathfrak{Re}\{e^n\}\|_{L^2_{-s}(\mathbb{R}^d)}
    \leq C\omega^{2s-2}
    {\mathcal N}_{\omega}(\beta^n)
\|f\|_{L^2_s(\mathbb{R}^d)},
\end{equation}
and similar for the 
$H^1_{-s}$ norm with the
additional assumption (H4).
The assumptions of Theorem \ref{thm:convergence} are precisely that assumptions (H1) and, (H4) for
the $H^1_{-s}$ case, hold, and we will show below that (H2) and (H3) also hold. 

To remove the $\omega$-dependency
of our norms of interest we will work with the rescaled transfer function $\bar{\beta}(r) := \beta(\omega r)$.
With this scaling 
${\mathcal N}_\omega(\beta^n)={\mathcal N}_1(\bar\beta^n)$, and we will demonstrate that for $n \in \mathbb{Z}^+$,
\begin{equation} \label{eq:N1Bound}
{\mathcal N}_1(\bar\beta^n)\leq M n^{-1/2},
\end{equation}
for some constant $M$ independent of $n$. The estimate (\ref{eq:errorBound}) will then imply that 
\begin{equation} \label{eq:finalErrorBound}
\|\mathfrak{Re}\{e^n\} \|_{L^2_{-s}(\mathbb{R}^d)} \leq C\omega^{2s-2}n^{-\frac{1}{2}}\|f\|_{L^2_s(\mathbb{R}^d)},    
\end{equation}
for some constant $C$, and similarly for the $H^1_{-s}$-norm, finally proving Theorem \ref{thm:convergence}. 

It remains to prove \eqref{eq:N1Bound}, which amounts to estimating
$$
\|\bar{\beta}^n\|_{L^\infty(\mathbb{R}\setminus I_\delta)},\qquad
\|\bar{\beta}^n\|_{L^1(\mathbb{R})}
\quad \text{and}\quad
\|D_1\bar{\beta}^n\|_{L^1(0,\delta)}.
$$
To this end we investigate the function $\bar\beta(r)$. It has been shown in \cite{Waveholtz2020} that $\bar{\beta}$ satisfies
\begin{equation} \label{betaEstimates}
|\bar{\beta}(r)|\leq 
\begin{cases}
    1 - \frac{1}{2}(r - 1)^2, & \left|r-1\right| \leq \frac{1}{2},
    \\
    \frac{1}{2} , &|r - 1| \geq \frac{1}{2},
    \\
    \frac{a}{r - 1},& r >\omega,
\end{cases}
\ \ 
\leq\ \  
\begin{cases}
    e^{-\frac{1}{2}(r - 1)^2}, & \left|r-1\right| \leq \frac{1}{2},
    \\
    \frac{1}{2} , &|r - 1| \geq \frac{1}{2},
    \\
    \frac{a}{r - 1},& r >\omega,
\end{cases}    
\end{equation}
where $a$ = $3/4\pi$, for the choice of kernel $K(t)$ as in (\ref{eq:kernel}). To simplify our estimate of the norm $\|D_1\bar{\beta}^n\|_{L^1(0, \delta)}$ we will choose $\delta = 1/4$. 

For the max-norm we then get
from \eqref{betaEstimates},
\begin{equation} \label{eq:betaMax}
\|\bar{\beta}^n\|_{L^\infty(\mathbb{R}\setminus I_\delta)}
\leq
\sup_{|1-r|\geq1/4}
|\bar{\beta}(r)|^n
\leq \left(\frac{31}{32}\right)^n.
\end{equation}
For the $L^1$-norm, the definition (\ref{eq:betaDef}) of $\beta$ shows that $\bar{\beta}$ is an even function, so that, again using \eqref{betaEstimates},
\begin{align}\label{eq:betaL1}
\|\bar{\beta}^n\|_{L^1(\mathbb{R})} &= 2\int_0^\infty |\bar{\beta}^n(r)|\,dr = 2\left(\int_0^\frac{1}{2}|\bar{\beta}^n|dr + \int_\frac{1}{2}^\frac{3}{2}|\bar{\beta}^n|dr + \int_\frac{3}{2}^\infty|\bar{\beta}^n|dr\right) \nonumber\\
&\leq 2\left(\int_0^\frac{1}{2}2^{-n}dr + \int_\frac{1}{2}^\frac{3}{2}e^{-\frac{n}{2}(r - 1)^2} dr + a^n\int_\frac{3}{2}^\infty\frac{1}{(r-1)^n}dr\right)\nonumber \\
&= 2\left(2^{-n-1} + 
n^{-\frac{1}{2}}\int_{\frac{\sqrt{n}}{2}}^{-\frac{\sqrt{n}}{2}}e^{-\frac{r^2}{2}} dr + \left(\frac{a}{2}\right)^n\frac{2}{n-1}\right)
\leq C n^{-\frac{1}{2}},
\end{align}
since $a/2=3/8\pi<1$. In order to estimate the norm $\|D_1\bar{\beta}^n\|_{L^1(0, \delta)}$ we write the function $\bar{\beta}$ on an alternative form. Integration of the definition of $\bar{\beta}(r)$ using the integral (\ref{eq:betaDef}) shows that
\[
\bar{\beta}(r) = \sinc(r+1) + \sinc(r-1) - \frac{1}{2}\sinc(r),
\]
where we use the convention $\sinc(x) = \sin(2\pi x)/2\pi x$. One can rewrite this expression to find the form
\[
\bar{\beta}(r) = \frac{1}{\pi}\sin(2\pi r)\Big(\frac{r}{r^2 - 1} - \frac{1}{4r}\Big),
\]
which means that
\[
\bar{\beta}(1+r) = \sinc(r)\Big(1 + \frac{r^2}{2r^2 + 6 r + 4}\Big).
\]
We now define the symmetric part of $\bar\beta$, centered around one,
as
\[
\bar{\beta}_{\rm sym}(r) = \sinc(r)\Big(1 + \frac{r^2}{2r^2 + 4}\Big),
\]
for which $\bar{\beta}_{\rm sym}(r)=\bar{\beta}_{\rm sym}(-r)$.
Since $|\sin(x)| \leq |x-x^3/3\pi|$
for $|x|\leq\pi/2$, we have 
that
\[
|\sinc(x)| \leq 1 - \frac{4\pi}{3}x^2,
\]
which gives an estimate of $\bar\beta_{\rm sym}$,
$$
|\bar{\beta}_{\rm sym}(r)|
\leq \left(1 - \frac{4\pi}{3}r^2\right)
\Big(1 + \frac{r^2}{4}\Big)
\leq e^{-\frac{4\pi}{3} r^2}
 e^{\frac{1}{4} r^2}
\leq e^{\left(-\frac{4\pi}{3}+\frac14\right) r^2}.
 $$
We finally let $w$ be the ratio
between $\bar\beta$ and $\bar\beta_{\rm sym}$,
$$
w(r) = \frac{\bar{\beta}(1+r)}{\bar{\beta}_{\rm sym}(r)}
= 1-r^3 Q(r), 
\qquad Q(r) = \frac{6}{(3r^2+4)(2r^3+6r+4)},
$$
where $Q$ is a smooth function
on $(-\delta,\delta)$.
One can check that when $|r|\leq \delta=1/4$, 
$$
|Q(r)| \leq  \frac{6}{4(4-6\delta)}\leq 1,
\qquad
|w'(r)| = 3r^2|Q(r)|+|r|^3|Q'(r)|
\leq c_\delta r^2,
$$
where $c_\delta=3+\delta \max_{|r|\leq \delta} |Q'(r)|$. Then
\begin{align}\label{eq:betaDL1}
\int_0^\delta |D_1 \bar{\beta}^n(r)|dr &=
\int_0^\delta \left|\frac{\bar{\beta}^{n}(1+r) - \bar{\beta}(1-r)}{r}\right|dr 
= \int_0^\delta|\bar{\beta}_{\rm sym}^n(r)|\Big|\frac{w(r)^n - w(-r)^n}{r}\Big|dr
\nonumber\\
&\leq \frac12\int_0^\delta |\bar{\beta}_{\rm sym}^n(r)|\sup_{|s| \leq r}n|w'(s)||w(s)|^{n-1}dr
\nonumber\\
&\leq \frac{c_\delta}2\int_0^\delta e^{\left(-\frac{4\pi}{3}+\frac14\right) nr^2}\sup_{|s| \leq r}ns^2|(1+|s|^3)|^{n-1}dr
\nonumber\\
&\leq \frac{c_\delta}2 n\int_0^\delta e^{\left(-\frac{4\pi}{3}+\frac14\right) nr^2}r^2e^{(n-1)r^2\delta}dr
\nonumber\\
&\leq \frac{c_\delta}2 n^{-1/2}\int_0^{\infty} e^{\left(-\frac{4\pi}{3}+\frac14+\delta\right) r^2}r^2dr = C n^{-1/2}.
\end{align}
Together, 
\eqref{eq:betaMax},
\eqref{eq:betaL1} and
\eqref{eq:betaDL1} prove \eqref{eq:N1Bound}. These estimates and \eqref{eq:SnBeta} also show that assumptions (H2) and (H3) indeed hold, justifying the argument made above. We can therefore conclude that the estimate (\ref{eq:finalErrorBound}) holds, which completes the proof of Theorem \ref{thm:convergence}. 

\section{Estimate of the Helmholtz Equation} \label{sec:HelmholtzEstimate}

In this section we prove a result which provides a bound for the weighted norms of solutions to the Helmholtz equation \eqref{eq:HEF} in terms of a weighted norm of the forcing function $F$. The result requires the following assumptions:
\begin{enumerate}[leftmargin=*,labelindent=1.5em]
    \item[(H1)] $s > \frac{3}{2}$ and $\omega \geq 1$.
    \item[(H2)] The forcing function $F$ satisfies  
    \[\hat{F}(\xi) = \hat{h}(|\xi|)\hat{f}(\xi)
    \] with $\hat{h}$ a real-valued function in $L^1(\mathbb{R})\cap L^\infty(\mathbb{R})$ and $f$ a real-valued function in $L^2_{s}(\mathbb{R}^d)$.
    \item[(H3)] There is some $0 < \delta < 1$ such that the function $\hat{h}$ satisfies
    \[
    \|D_\omega \hat{h}\|_{L^1(0, \delta\omega)} < \infty,
    \]
    where $D_\omega g$ denotes the function
    \[
    D_\omega g(x) = \frac{g(\omega+x) - g(\omega - x)}{2x}.
    \]
\end{enumerate}

The estimate comes in two forms, one for the case where assumptions (H1) through (H3) hold, and another for the case where the following stronger regularity condition on $f$ also holds:
\begin{enumerate}[leftmargin=*,labelindent=1.5em]
    \item[(H4)] $f \in H^1_s(\mathbb{R}^d)$.
\end{enumerate}
In what follows we let $I_\delta = (1-\delta, 1+\delta)$, so that $\omega I_\delta = (\omega - \omega \delta, \omega + \omega\delta)$. Using this notation we have
\begin{theorem} \label{thm:HEBound}
    Suppose the assumptions (H1)-(H3) given above hold. 
    Let
    $$
    {\mathcal N}_\omega(\hat{h})
:=\omega^{-1}\|\hat{h}\|_{L^1(\mathbb{R})} + \|D_\omega\hat{h}\|_{L^1(0, \omega\delta)}
+ \|\hat{h}\|_{L^\infty(\mathbb{R}\setminus \omega I_\delta)}.
    $$
    Then if $\alpha > 0$ and $\omega \geq 1$, the solution $u$ to 
    \begin{equation} \label{eq:HEF}
    \Delta u + \omega^2u +i\alpha\omega u = F,    
    \end{equation}
    satisfies
    \begin{equation*}
    \|\mathfrak{Re}\{u\}\|_{L^2_{-s}(\mathbb{R}^d)} \leq C\omega^{2s-2}{\mathcal N}_\omega(\hat{h})\|f\|_{L^2_s(\mathbb{R}^d)},    
    \end{equation*}
    for some constant $C$ independent of $\alpha$ and $\omega$, but dependent on, $\delta$, $d$, and $s$. If in addition assumption (H4)
    holds, then also
    \begin{equation*}
        \|\mathfrak{Re}\{u\}\|_{H^1_{-s}(\mathbb{R}^d)} \leq C\omega^{2s-1}{\mathcal N}_\omega(\hat{h})\|f\|_{H^1_s(\mathbb{R}^d)},
    \end{equation*}
    for the same constant $C$. The same bounds hold for outgoing solutions in the case $\alpha = 0$. 
\end{theorem}

\begin{remark}
    Agmon shows a
    similar estimate in \cite{Agmon}, 
    for general $F\in L^{2}_s({\mathbb R}^d)$,
    namely that for $s > \frac{1}{2}$ one can find $C \in \mathbb{R}$ depending on $\omega^2$, $d$, and $s$ such that
    \[
    \|u\|_{H^2_{-s}(\mathbb{R}^d)} \leq C\|F\|_{L^2_s(\mathbb{R}^d)}
    \]
    for any outgoing solution $u$ satisfying the Helmholtz problem 
    \[
    \Delta u + \omega^2u = F.
    \]
    With $F$ of the type considered
    in Theorem~\ref{thm:HEBound},
    this result, along with the error equation of Proposition \ref{prop:errorEq} produces
    \[
    \|e^n\|_{H^2_{-s}(\mathbb{R}^d)} \leq \|\mathcal{S}^nf\|_{L^2_s(\mathbb{R}^d)},
    \]
    an expression one might expect to be sufficient to show convergence of the Waveholtz method. However, the right-hand side of this relation does not converge to zero in
    general, as one can see for instance in the case $d = s = 1$, where
    \[
\|\mathcal{S}^nf\|_{L^2_1(\mathbb{R})} = \|\mathcal{S}^nf\langle x \rangle\|^2_{L^2(\mathbb{R})}=\|\mathcal{S}^nf\|^2_{L^2(\mathbb{R})} + \big\||x|^2\mathcal{S}^nf\big\|^2_{L^2(\mathbb{R})}.
    \]
    The first of these terms converges
    to zero since
    the $L^2$-norm of $\beta^n$ converges
    to zero,
    \[
    \|\mathcal{S}^nf\|_{L^2(\mathbb{R})} = \|\beta^n(|\xi|)\hat{f}(\xi)\|_{L^2(\mathbb{R})} \leq \|\beta^n\|_{L^2(\mathbb{R})}\|\hat{f}\|_{L^\infty(\mathbb{R})}\to 0.
    \]
    When investigating the second term $\big\||x|^2\mathcal{S}^nf\big\|_{L^2(\mathbb{R})}$, however, we find that it
    does not approach zero as $n \to \infty$. 
    Its norm can be understood in terms of the norm of derivatives of the Fourier transform,
    \[
    \big\||x|^2\mathcal{S}^nf\big\|_{L^2(\mathbb{R})}= \|\partial_\xi^2\big(\beta^n(|\xi|)\hat{f}(\xi)\big)\|_{L^2(\mathbb{R})}.
    \]
    The fundamental difference from the first term
    involving just $\mathcal{S}^nf$ is that the $L^p$-norms of the derivatives of $\beta^n$ do not approach zero as $n\to \infty$. 
        The same thing happens
    with other choices of $s$.
    This stands in contrast to what we have seen above, namely that $\|D_\omega\beta^n\|_{L^1(0, \omega\delta)} \to 0$ with $n$. 
\end{remark}

\begin{subsection}{
Proof of Theorem~\ref{thm:HEBound}}

The proof is given in four steps. First, in Section~\ref{sec:ThmProofStep1}, we consider just $\omega=1$ and $f\in L^1_{s}(\mathbb{R}^d)$, proving Theorem~\ref{thm:HEBound} for this case. Second, we extend the proof to the case when $f\in H^1_{s}(\mathbb{R}^d)$ in Section~\ref{sec:ThmProofStep2}. Third, in Section~\ref{sec:hfest}, we use scaling arguments to convert the proof for $\omega=1$ to a proof with general $\omega\geq 1$. This gives frequency-explicit estimates. Last, we prove the final statement in the theorem about outgoing solutions in Section~\ref{sec:ThmProofStep4}.

\begin{subsubsection}{An estimate for the case $\omega = 1$}
\label{sec:ThmProofStep1}

Suppose that 
\[
\Delta u + u + i\alpha u = F
\]
where $F$ and $\alpha$ is as described in Theorem \ref{thm:HEBound}
with $\omega=1$. The plan is to find a bound for some weighted norm of $\mathfrak{Re}\{u\}$ by the following observation.
Let $(\,\cdot\,,\,\cdot\,)$
be the usual $L^2$ inner product.
Suppose we found an estimate of the type $\big|(\mathfrak{Re}\{u\}, v)\big| \leq M\|v\|_{L^2_s(\mathbb{R}^d)}$, where
$v$ is real-valued and  $M$ is some constant independent of $v$. 
Choosing $v = \langle x\rangle^{-2s}\mathfrak{Re}\{u\}$ would then produce a bound for a weighted norm of $\mathfrak{Re}\{u\}$:
\begin{equation} \label{eq:estimateReU}
    \|\mathfrak{Re}\{u\}\|_{L^2_{-s}(\mathbb{R}^d)}^2 
    =
    \big|(\mathfrak{Re}\{u\}, \mathfrak{Re}\{u\}\langle x \rangle^{-2s})\big| 
    \leq M\|v\|_{L^2_s(\mathbb{R}^d)} = M\|\mathfrak{Re}\{u\}\|_{L^2_{-s}},
\end{equation}
so that in turn $\|\mathfrak{Re}\{u\}\|_{L^2_{-s}} \leq M$. Such an estimate is our goal. We suppose therefore $v$ to be real-valued and sufficiently regular to Fourier transform. Then $(\mathfrak{Re}\{u\}, v) = \mathfrak{Re}\{(u, v)\} = \mathfrak{Re}\{(\hat{u}, \hat{v})\}$, 
and an estimate of the type \eqref{eq:estimateReU} could be found by investigating this last inner product.
We therefore start by computing 
the Fourier transform of the 
Helmholtz solution $u$,
\[
\hat{u}(\xi) = \frac{\hat{F}(\xi)}{1 - |\xi|^2 + i\alpha},
\]
which permits us to write
\begin{equation} \label{eq:uvInnerProduct}
(\hat{u}, \hat{v}) = \int_{\mathbb{R}^d}\frac{\hat{F}\hat{v}^*(\xi)}{1 - |\xi|^2 + i\alpha}d\xi.
    \end{equation}
Introducing the
set $\Omega_\delta =\{x\in{\mathbb R}^d
: \bigl||\xi|-1\bigr|< \delta\}$
we divide the integral in \eqref{eq:uvInnerProduct}
into three parts 
$(\hat{u}, \hat{v}) =I_1+I_2+I_3$,
where
\[
I_1 := \int_{{\mathbb R}^d\setminus\Omega_\delta}\frac{\hat{F}\hat{v}^*(\xi)}{1 - |\xi|^2 + i\alpha}d\xi,
\]
and
\[
I_2 := \int_{\Omega_\delta}\frac{\hat{F}\hat{v}^*(\xi)}{1 - |\xi|^2 + i\alpha}d\xi-I_3,
\qquad
I_3 :=  \int_{\Omega_\delta}\frac{\hat{F}\hat{v}^*(\xi)}{-2(|\xi|-1) + i\alpha}d\xi.
\]
We will next find upper bounds for the absolute value of each of these integrals in turn, which will allow us to find an upper bound for $|\mathfrak{Re}\{(\hat{u}, \hat{v})\}|$. 
 
We start with $I_1$.
Since $|1 - |\xi|^2|\geq \delta$ outside $\Omega_d$, we have
\begin{align}
    |I_1| &\leq 
    \frac{1}{\delta}
    \int_{{\mathbb R}^d\setminus\Omega_\delta}|
    \hat{h}(|\xi|)\hat{f}(\xi)\hat{v}(\xi)|d\xi
    \leq 
    \frac{1}{\delta}
    ||\hat{h}||_{L^{\infty}({\mathbb R}\setminus I_\delta)}
    \int_{{\mathbb R}^d\setminus\Omega_\delta}|\hat{f}(\xi)\hat{v}(\xi)|d\xi
    \nonumber
    \\
    &\leq 
    \frac{1}{\delta}
    ||\hat{h}||_{L^{\infty}({\mathbb 
    R}\setminus I_\delta)}
    ||\hat{f}||_{L^{2}({\mathbb 
    R}^d)}
    ||\hat{v}||_{L^{2}({\mathbb 
    R}^d)}
=    \frac{1}{\delta}
    ||\hat{h}||_{L^{\infty}({\mathbb 
    R}\setminus I_\delta)}
    ||f||_{L^{2}({\mathbb 
    R}^d)}
    ||{v}||_{L^{2}({\mathbb 
    R}^d)}.
\label{eq:I_1Result}
\end{align}
For $I_2$ and $I_3$ we will use spherical coordinates and therefore introduce some new notation. Let
\[
g(r) = \begin{cases}
\hat{f}(r)\hat{v}^*(r) + \hat{f}(-r)\hat{v}^*(-r), &d = 1,
\\
\int_{|\eta| = 1} \hat{f}(r\eta)\hat{v}^*(r\eta) d\eta,  &d > 1.
\end{cases}
\]
Then we can write
\[
I_2 := \int_{I_\delta}\frac{\hat{h}(r)g(r)}{1 - r^2 + i\alpha}r^{d-1}dr - I_3,\qquad
I_3 := \int_{I_\delta} \frac{\hat{h}(r)g(r)}{-2(r - 1) + i\alpha}r^{d-1}dr.
\]
Before estimating these integrals,
we prove a Lemma about the function $g(r)$.

\begin{lemma} \label{lemma:gLemma}
    The following holds for the function $g$ defined above.
    \begin{enumerate}[label=(\roman*),leftmargin=*]
        \item $g$ is real-valued if $f$ and $v$ are real-valued.
        \item If $s > \frac{1}{2}$, then $$|g(r)| \leq {K}_1 r^{-d}\langle r\rangle^{2s}\|f\|_{L^2_s(\mathbb{R}^d)}\|v\|_{L^2_s(\mathbb{R}^d)},$$ with $K_1$ independent of $r$, $v$, and $f$, but dependent on $d$ and $s$.
        \item If $s > \frac{1}{2}$, then $$|g'(r)| \leq {K}_2r^{-d-1}\langle r\rangle^{2s+2}\|f\|_{L^2_{s+1}(\mathbb{R}^d)}\|v\|_{L^2_{s+1}(\mathbb{R}^d)},$$ again with $K_2$ independent of $r$, $v$, and $f$, but dependent on $d$ and $s$.
    \end{enumerate}
\end{lemma}
\begin{proof}
We first note that this function $g$ is indeed real-valued, due to the assumptions that $f$ and $v$ are real-valued. In the case $d = 1$ we have 
\[
g^*(r) = \hat{f}^*(r)\hat{v}(r) + \hat{f}^*(-r)\hat{v}(-r) = \hat{f}(-r)\hat{v}^*(-r) + \hat{f}(r)\hat{v}^*(r) = g(r),
\]
since the assumption that $f$ and $v$ are real-valued means that $\hat{f}^*(r) = \hat{f}(-r)$, and similarly for $\hat{v}$. If $d > 1$ we have, by symmetry of the sphere,
\begin{align*}
g(r) &= \frac12\int_{|\eta| = 1} \hat{f}(r\eta)\hat{v}^*(r\eta) 
+\hat{f}(-r\eta)\hat{v}^*(-r\eta) 
d\xi \\
&=\frac12\int_{|\eta| = 1} \hat{f}(r\eta)\hat{v}^*(r\eta) 
+\hat{f}(r\eta)^*\hat{v}(r\eta) 
d\eta
=\int_{|\eta| = 1} \mathfrak {Re}\{\hat{f}(r\eta)\hat{v}^*(r\eta)\}
d\eta.
\end{align*}
This proves {\it (i)}.
We now estimate $|g(r)|$. 
In the case $d > 1$ we denote by $B^d$ the set of points $\xi$ in $\mathbb{R}^d$ such that $|\xi| < 1$, and by $\tau : B^d \longrightarrow \partial B^d$ the trace map, so that 
\begin{align} \label{eq:gInfEst}
    |g(r)| \leq \int_{|\eta| = 1} |\hat{f}\hat{v}^*(r\eta)|d\eta \leq  \|\tau\hat{f}(r\,\cdot)\|_{L^2(\partial B^d)}\|\tau\hat{v}(r\,\cdot)\|_{L^2(\partial B^d)}.
\end{align}
We can then use Proposition \ref{prop:traceEst} to find 
\begin{align*}
    \|\tau\hat{f}(r\,\cdot)\|_{L^2(\partial B^d)} \leq \|\tau \hat{f}(r\,\cdot) \|_{H^{s-\frac{1}{2}}(\partial B^d)} \leq K\|\hat{f}(r\,\cdot)\|_{H^{s}(B^d)} \leq K\|\hat{f}(r\,\cdot)\|_{H^{s}(\mathbb{R}^d)},
\end{align*}
provided $s > \frac{1}{2}$. Now using the convention \eqref{eq:SobolevWeightedLpEquiv} we see that $\|\ \widehat{\cdot}\ \|_{H^s(\mathbb{R}^d)} = \| \cdot \|_{L^2_s(\mathbb{R}^d)}$ and 
\[
\hat{f}(r\,\cdot)= r^{-d}\widehat{f\left(\frac{\cdot}{r}\right)},
\]
we see that 
\[
\|\tau\hat
{f}(r\,\cdot)\|_{L^2(\partial B^d)} \leq K\|\hat{f}(r\,\cdot)\|_{H^s(\mathbb{R}^d)} = Kr^{-d}\left\|f\left(\frac{\cdot}{r}\right)\right\|_{L^2_s(\mathbb{R}^d)} \leq Kr^{-\frac{d}{2}}\langle r\rangle^s
\|f\|_{L^2_s(\mathbb{R}^d)},
\]
and similar for $\|\tau\hat
{v}(r\,\cdot)\|_{L^2(\partial B^d)}$, which in light of \eqref{eq:gInfEst} means that 
\begin{align*}
    |g(r)| &\leq K^2 r^{-d}\langle r\rangle^{2s}\|f\|_{L^2_s(\mathbb{R}^d)}\|v\|_{L^2_s(\mathbb{R}^d)}.
\end{align*}
For $d = 1$, we instead note that $H^s(B^1)$ is continuously embedded in $C^0(B^1)$ for $s > \frac{1}{2}$; see \cite[Theorem 9.8]{LionsMagenes}. This means that 
\begin{align*}
    |g(r)| &\leq 2\|\hat{f}(r\,\cdot)\|_{L^\infty(B^1)}\|\hat{v}(r\,\cdot)\|_{L^\infty(B^1)} 
    \\
    &\leq \bar{K}r^{-2}\Big\|f\Big(\frac{\cdot}{r}\Big)\Big\|_{L^2_s(\mathbb{R})}\Big\|v\Big(\frac{\cdot}{r}\Big)\Big\|_{L^2_s(\mathbb{R})} \leq \bar{K}r^{-1}\langle r\rangle^{2s}\|f\|_{L^2_s(\mathbb{R})}\|v\|_{L^2_s(\mathbb{R})},
\end{align*}
for some $\bar{K}$ and such $s$, by arguments similar to the ones used in the higher-dimensional case. 
Denoting by ${K}_1$ the constants found in the estimates above for the different values of $d$, we have proved {\it (ii)}.
Finally, for {\it (iii)} we similarly have
\begin{align*}
    |g'(r)| &= \Big| \partial_r\int_{|\eta| = 1} \hat{f}(r\eta)\hat{v}^*(r\eta)d\eta \Big|
    \\
    &\leq 
    \frac1r
    \int_{|\eta| = 1} \big|\eta\cdot\nabla_\eta\hat{f}(r\eta)\big|\big|\hat{v}^*(r\eta)\big|d\eta + 
    \frac1r
    \int_{|\eta| = 1} \big|\hat{f}(r\eta)
    \big|\big|\eta\cdot\nabla_\eta
    \hat{v}^*(r\eta)\big|d\eta,
\end{align*}
and these integrals can be estimated in the same way as \eqref{eq:gInfEst}, producing
\begin{align*}
    |g'| &\leq \frac{K^2}{r}\big( \|\nabla_\eta\hat{f}(r\,\cdot)\|_{H^s(\mathbb{R}^d)}\|\hat{v}^*(r\,\cdot)\|_{H^s(\mathbb{R}^d)} 
+ \|\hat{f}(r\,\cdot)\|_{H^s(\mathbb{R}^d)}
\|\nabla_\eta\hat{v}^*(r\,\cdot)\|_{H^s(\mathbb{R}^d)}\big)
    \\
    &\leq \frac{2K^2}{r} \|\hat{f}(r\,\cdot)\|_{H^{s+1}(\mathbb{R}^d)}\|\hat{v}(r\,\cdot)\|_{H^{s+1}(\mathbb{R}^d)}
    \leq 2K^2r^{-d-1}\langle r\rangle^{2s+2}\|f\|_{L^2_{s+1}(\mathbb{R}^d)}\|v\|_{L^2_{s+1}(\mathbb{R}^d)},
\end{align*}
for $s > \frac{1}{2}$. Again use  \cite[Theorem 9.8]{LionsMagenes} in the one-dimensional case, this time for the function $g'(r)$, producing
\begin{align*}
|g'(r)| &\leq \frac1r \|\partial_\eta\hat{f}(r\,\cdot)\|_{L^\infty(B^1)}\|\hat{v}(r\,\cdot)\|_{L^\infty(B^1)} + \frac1r \|\hat{f}(r\,\cdot)\|_{L^\infty(B^1)}\|\partial_\eta\hat{v}(r\,\cdot)\|_{L^\infty(B^1)} 
\\
&\leq \bar{K}r^{-2}\langle r\rangle^{2s+s}\|f\|_{L^2_{s+1}(\mathbb{R})}\|v\|_{L^2_{s+1}(\mathbb{R})},
\end{align*}
by an argument near-identical to the one for $d > 1$. Denoting by ${K}_2$ the constants found above, we have shown the estimate in the statement of the lemma, for all $d$. 
\end{proof}
\noindent We now turn our attention to estimating the remaining integrals $I_2$, and $I_3$.
For $I_2$ we recall that
\[
I_2 = \int_{I_\delta}\Big(\frac{\hat{h}(r)g(r)}{1 - r^2 + i\alpha} - \frac{\hat{h}(r)g(r)}{-2(r - 1) + i\alpha} \Big)r^{d-1}dr,
\]
and investigate how we can simplify this fraction. We have
\begin{align*}
    \left|\frac{1}{1 -r^2 + i\alpha} - \frac{1}{-2(r - 1) + i\alpha} \right| &= \left| \frac{-2(r - 1) + i\alpha - \big(1 - |r^2 + i\alpha\big)}{(1 - r^2 + i\alpha)(-2(r - 1) + i\alpha)}\right|
    \\
    &\leq \frac{1}{2}\frac{(1 - r)^2}{\big|1 - r^2\big|\big|r - 1\big|} = \frac{1}{2}\frac{1}{r + 1}.
\end{align*}
This means that by also using Lemma~\ref{lemma:gLemma},
\begin{align}
    |I_2| &\leq \int_{I_\delta} \frac{|\hat{h}(r)g(r)|}{2(r + 1)}r^{d-1}dr \leq {K}_1\|f\|_{L^2_s(\mathbb{R}^d)}\|v\|_{L^2_s(\mathbb{R}^d)}\int_{I_\delta}|\hat{h}(r)|\frac{
    \langle r\rangle^{2s}}{r(r+1)}dr \nonumber
    \\
    &\leq B\|f\|_{L^2_s(\mathbb{R}^d)}\|v\|_{L^2_s(\mathbb{R}^d)}\|\hat{h}\|_{L^1(\mathbb{R})},  \label{eq:I_2Result}   
\end{align}
where $B \in \mathbb{R}$ is some constant which is independent of $\alpha$, but depends on $\delta$ and $s$. Finally, for $I_3$ we note that 
    \[
    \mathfrak{Re}\{I_3\} = \mathfrak{Re}\Big\{\int_{I_\delta} \frac{\hat{h}(r)g(r)}{-2(r - 1) + i\alpha}r^{d-1}dr\Big\} = -\frac{1}{2}\int_{1-\delta}^{1+\delta}\frac{\hat{h}(r)g(r)(r-1)}{(r-1)^2 + \frac
    {\alpha^2}{4}}r^{d-1}dr, 
    \]
    since $\hat{h}$ and $g$ are real-valued. This is due to (H2) and Lemma \ref{lemma:gLemma}, as $f$ and $v$ are real-valued. After a suitable change of variables this can be written as
    \[
    \mathfrak{Re}\{I_3\} = -\frac{1}{2}\int_0^\delta \frac{(1+r)^{d-1}\hat{h}(1+r)g(1 + r) - (1-r)^{d-1}\hat{h}(1-r)g(1-r)}{r^2 + \frac{\alpha^2}{4}}rdr.
    \]
    We find
    \begin{align}
    |\mathfrak{Re}\{I_3\}| &\leq \int_0^\delta \bigg|\frac{(1+r)^{d-1}\hat{h}(1+r)g(1 + r) - (1-r)^{d-1}\hat{h}(1-r)g(1-r)}{2r}\bigg|dr \nonumber
    \\
    &\leq \int_0^\delta \big|(1+r)^{d-1}g(1+r)D_1\hat{h}(r)\big| dr \nonumber
    \\
    &\qquad\qquad\qquad+ \int_0^\delta \big|\hat{h}(1-r)g(1+r)D_1\big(t^{d-1}\big)(r)\big| dr \nonumber
    \\
    &\qquad\qquad\qquad\qquad\qquad+ \int_0^\delta\big|\hat{h}(1-r)(1-r)^{d-1}D_1g(r)\big|dr  \nonumber\\
    &\leq
    C\|g\|_{L^\infty(I_\delta)} 
    \int_0^\delta 
    \left(\big|D_1\hat{h}(r)\big| + \big|\hat{h}(1-r)\big|\right) dr \nonumber
    \\
    &\qquad\qquad\qquad\qquad\qquad+ 
    \|D_1g(r)\|_{L^\infty(0, \delta)} 
    \int_0^\delta\big|\hat{h}(1-r)\big|dr.  
    \label{eq:I_3estimate}
    \end{align}
    Each of these terms can be estimated using the results in Lemma \ref{lemma:gLemma}. In particular, noting that if $s>1/2$, then
    \begin{align*}
        \|g\|_{L^\infty( I_\delta)} &= \sup_{r\in I_\delta}|g(r)| \leq \sup_{r\in I_\delta}{K}_1 r^{-d}\langle r\rangle^{2s}\|f\|_{L^2_s(\mathbb{R}^d)}\|v\|_{L^2_s(\mathbb{R}^d)} 
        \\
        \|D_1g\|_{L^\infty(0, \delta)} &\leq \|g'(r)\|_{L^\infty(I_\delta)} \leq \sup_{r\in I_\delta}\,{K}_2r^{-d-1}\langle r \rangle^{2s+2}\|f\|_{L^2_{s+1}(\mathbb{R})}\|v\|_{L^2_{s+1}(\mathbb{R})},
    \end{align*}
    and \eqref{eq:I_3estimate} tells us that for such $s$
    \begin{equation} \label{eq:I_3Result}
        |\mathfrak{Re}\{I_3\}| \leq \big(C_1\|\hat{h}\|_{L^1(\mathbb{R})} + C_2\|D_1\hat{h}\|_{L^1(0, \delta)}\big)\|f\|_{L^2_{s+1}(\mathbb{R}^d)}\|v\|_{L^2_{s+1}(\mathbb{R}^d)},
    \end{equation}
    where $C_1, C_2 \in \mathbb{R}^+$ are some real numbers independent of $\alpha$, but dependent on $\delta$, $s$ and $d$. Now, bringing together the results \eqref{eq:I_1Result}, \eqref{eq:I_2Result}, \eqref{eq:I_3Result}, and denoting by $C$ the maximum over the constants in these results, produces
    \begin{align*}
    |\mathfrak{Re}\{(\hat{u}, \hat{v})\}| &\leq C\big(\|\hat{h}\|_{L^1(\mathbb{R})} + \|D_1\hat{h}\|_{L^1(0, \delta)}    +\|\hat{h}\|_{L^\infty(\mathbb{R}\setminus I_\delta)} 
    \big)\|f\|_{L^2_{s'}(\mathbb{R}^d)}\|v\|_{L^2_{s'}(\mathbb{R}^d)},
    \\ &=
    C {\mathcal N}_1(\hat{h})
    \|f\|_{L^2_{s'}(\mathbb{R}^d)}\|v\|_{L^2_{s'}(\mathbb{R}^d)},
    \end{align*}
    when $s'>3/2$. Choosing $v(x) = \langle x \rangle^{-2s'}\mathfrak{Re}\{u\}(x)$ then yields 
\begin{equation*}
\|\mathfrak{Re}\{u\}\|^2_{L^2_{-s'}(\mathbb{R}^d)} \leq C{\mathcal N}_1(\hat{h})\|f\|_{L^2_{s'}(\mathbb{R}^d)}\|\mathfrak{Re}\{u\}\|_{L^2_{-s'}(\mathbb{R}^d)}.
\end{equation*}
This implies
\[
\|\mathfrak{Re}\{u\}\|_{L^2_{-s}(\mathbb{R)}} \leq C{\mathcal N}_1(\hat{h})\|f\|_{L^2_{s}(\mathbb{R}^d)},
\]
which proves Theorem \ref{thm:HEBound} for $\omega = 1$ in the case where assumptions (H1) through (H3) hold.

\end{subsubsection}

\subsubsection{
Extension to the $H^1_s(\mathbb{R}^d)$-norm}
\label{sec:EstimateH1}

\label{sec:ThmProofStep2}

We now consider the case where in addition to the assumptions (H1)--(H3), we add the assumption (H4) that $f \in H^1_s(\mathbb{R}^d)$. Suppose again that $u$ is the unique solution of
\[
\Delta u + u +i\alpha u= F.
\]
We then consider the function
$u_j = \partial_{x_j}u$
which is the unique solution to
\begin{equation} \label{eq:newHE}
\Delta u_j + u_j +i\alpha u_j = \partial_{x_j}F.
\end{equation}
Moreover, since $\hat{F}(\xi) = \hat{h}(|\xi|)\hat{f}(\xi)$, we get
\[
\widehat{\partial_{x_j}F}(\xi) = i\xi_j\hat{h}(|\xi|)\hat{f}(\xi) = \hat
h(|\xi|)\big(i\xi_j\hat{f}(\xi)\big)
= h(|\xi|)\widehat{\partial_{x_j}f}(\xi).
\]
We can then apply Theorem \ref{thm:HEBound} to equation $(\ref{eq:newHE})$ to find
\begin{align*}
    \|\mathfrak{Re}\{\partial_{x_j}u\}\|_{L^2_{-s}(\mathbb{R}^d)} &= \|\partial_{x_j}\mathfrak{Re}\{u\}\|_{L^2_{-s}(\mathbb{R}^d)}
    \leq C{\mathcal N}_1(\hat{h})\|\partial_{x_j}f\|_{L^2_s(\mathbb{R}^d)}, 
\end{align*}
provided that $\partial_{x_j}f \in L^2_s(\mathbb{R}^d)$. Since we have assumed $f \in H^1_s(\mathbb{R}^d)$ we find
\begin{align*}
    \|\mathfrak{Re}\{u\}\|_{H^1_{-s}(\mathbb{R}^d)} &= \Big(\sum_{|\alpha| \leq 1}\|\partial^\alpha\mathfrak{Re}\{u\}\|^2_{L^2_{-s}(\mathbb{R}^d)}\Big)^{\frac{1}{2}}
\leq C{\mathcal N}_1(\hat{h})\|f\|_{H^1_s(\mathbb{R}^d)}.
    \end{align*}
This shows the bound in the case that assumptions (H1) through (H4) hold.

\begin{subsubsection}{The case of $\omega \geq 1$} \label{sec:hfest}

    We aim now to extend the above result to the case of general $\omega\geq 1$ in a manner that makes it clear how the estimate scales with increasing $\omega$. Suppose therefore that $u$ solves the equation
    \begin{equation} \label{eq:HEnonzeroOmega}
    \Delta u + \omega^2u + i\alpha\omega u = F,    
    \end{equation}
    for some $\omega \geq 1$ and that assumptions (H1) through (H3) hold for this $F$. Then defining $\tilde{u}(x) := u(\omega^{-1}x)$, $\tilde{\alpha} := \omega^{-1}\alpha$, and $G(x) := \omega^{-2}F(\omega^{-1}x)$ produces 
    \begin{align*}
        \Delta \tilde{u}(x) + \tilde{u}(x) + i\tilde{\alpha}\tilde{u}(x) &= \omega^{-2}\big(\Delta u(\omega^{-1}x) + \omega^2u(\omega^{-1}x) + i\alpha \omega u(\omega^{-1}x)\big) 
        \\
        &= \omega^{-2}F(\omega^{-1}x) = G(x).
    \end{align*}
    That is, the function $\tilde{u}(x)$ satisfies
    the Helmholtz equation \eqref{eq:HEF} with $\omega=1$ and
    \[
    \hat{G}(\xi) = \omega^{d-2}\hat{F}(\omega\xi) = \omega^{d-2}\hat{h}(\omega|\xi|)\hat{f}(\omega\xi) = \hat{k}(|\xi|)\hat{g}(\xi),
    \]
    where we defined $\hat{k}(|\xi|) = \hat{h}(\omega|\xi|)$ and $\hat{g}(\xi) = \omega^{d-2}\hat{f}(\omega\xi)$, or equivalently $g(x) = \omega^{-2}f(\omega^{-1}x)$. Because of the assumptions made on the function $F$, the function $G$ also satisfies assumptions (H1) through (H3), as $\|D_1\hat{k}\|_{L^1(0, \delta)} = \|D_\omega\hat{h}\|_{L^1(0, \omega\delta)}$.
    In fact, $
    {\mathcal N}_{1}(\hat{k})
    =
    {\mathcal N}_{\omega}(\hat{h})$
    since
    $$
    \|\hat{k}\|_{L^1(\mathbb{R})} = \omega^{-1}\|\hat{h}\|_{L^1(\mathbb{R})},
    \qquad
    \|\hat{k}\|_{L^\infty(\mathbb{R}\setminus I_\delta)} = \|\hat{h}\|_{L^\infty(\mathbb{R}\setminus \omega I_\delta)}.
    $$
In light of the above calculation, 
and since it is proven for $\omega=1$,
we can therefore apply Theorem \ref{thm:HEBound} to the function $\tilde{u}$, producing the estimate
    \[
    \|\mathfrak{Re}\{\tilde{u}\}\|_{L^2_{-s}(\mathbb{R}^d)} \leq 
    C{\mathcal N}_{\omega}(\hat{h})\|g\|_{L^2_{s}(\mathbb{R}^d)},
    \]
    with $C$ independent of $\omega$ and $\alpha$. Now since $\omega \geq 1$, 
    \begin{equation*}
        \|\mathfrak{Re}\{\tilde{u}\}\|_{L^2_{-s}(\mathbb{R}^d)} = \Big(\omega^d\int_{\mathbb{R}^d}|\mathfrak{Re}\{u\}|^2\langle\omega x\rangle^{-2s}dx\Big)^{\frac{1}{2}} \geq \omega^{\frac{d}{2}-s}\|\mathfrak{Re}\{u\}\|_{L^2_{-s}(\mathbb{R}^d)},
    \end{equation*}
    as for $\omega \geq 1$ we have $\omega\langle x\rangle \geq \langle \omega x\rangle$, and by assumption $s > \frac{1}{2}$. For those same reasons we find 
    \begin{equation}\label{eq:gest}
    \|g\|_{L^2_{s}(\mathbb{R}^d)} = \Big(\omega^{d-4}\int_{\mathbb{R}^d}|f(x)|^2\langle\omega x\rangle^{2s}dx\Big)^{\frac{1}{2}} \leq \omega^{\frac{d}{2}+s-2}\|f\|_{L^2_s(\mathbb{R}^d)}, 
    \end{equation}
    so that we finally get
    \begin{align*}
        \|\mathfrak{Re}\{u\}\|_{L^2_{-s}(\mathbb{R}^d)} &\leq \omega^{s-\frac{d}{2}}\|\mathfrak{Re}\{\tilde{u}\}\|_{L^2_{-s}(\mathbb{R}^d)} 
        \leq C\omega^{2s-2}
        {\mathcal N}_{\omega}(\hat{h})\|f\|_{L^2_{s}(\mathbb{R}^d)},
    \end{align*}
    where $C$ is independent of $\omega$ and $\alpha$. In the case that we impose on $F$ also the assumption (H4), we can similarly find
    \begin{align*}
        \|\mathfrak{Re}\{\tilde{u}\}\|_{H^1_{-s}(\mathbb{R}^d)} &= \Big(\sum_{|\alpha| \leq 1} \|\omega^{-|\alpha|}\partial^\alpha\mathfrak{Re}\{u(\omega^{-1}x)\}\|_{L^2_{-s}(\mathbb{R}^d)}^2 \Big)^{\frac{1}{2}}
        \\
        &\geq \omega^{\frac{d}{2}-s-1}\|\mathfrak{Re}\{u\}\|_{H^1_{-s}(\mathbb{R}^d)},
    \end{align*}
    again using the assumption that $\omega \geq 1$. A similar calculation produces the corre\-sponding estimate
    \[
    \|g\|_{H^1_s(\mathbb{R}^d)} \leq \omega^{\frac{d}{2}+s-2}\|f\|_{H^1_s(\mathbb{R}^d)},
    \]
    so that the discussion in Section \ref{sec:EstimateH1} gives 
    \begin{equation}
    \label{eq:omegaDepH1}
        \|\mathfrak{Re}\{u\}\|_{H^1_{-s}(\mathbb{R}^d)} \leq C\omega^{2s-1}
        {\mathcal N}_{\omega}(\hat{h})\|f\|_{H^1_{s}(\mathbb{R}^d)}.
    \end{equation}
    This generalises Theorem \ref{thm:HEBound} to the case $\omega \geq 1$.
\end{subsubsection}

\begin{subsubsection}{The case of outgoing solutions}
\label{sec:ThmProofStep4}

Finally we turn our attention to the case of outgoing solutions $u$ to equation \eqref{eq:HEnonzeroOmega}. We will first consider the case where assumptions (H1) through (H3) hold, as the case where (H4) holds can be done entirely analogously.
    Let $z_\alpha = \omega^2 + i\alpha\omega$, and denote by $R(z)$ the resolvent operator $(-\Delta - z)^{-1}$, which is is known to be well-defined for $z$ not on the positive real axis. Recall Proposition \ref{prop:LAP}, which states that the limit
    \[
    \lim_{\alpha \rightarrow 0^+} R(z_\alpha) =: R(\omega^2)
    \]
    exists as an element of the space of bounded linear operators from $L^{2}_s(\mathbb{R}^d)$ to $H^2_{-s}(\mathbb{R}^d)$, endowed with the uniform operator topology, provided that $s>\frac{1}{2}$. We get for $u_\alpha := -R(z_\alpha)F$ and $u_0 := -R(\omega^2)F$, since 
    $\|\cdot\|_{L^2_{-s}(\mathbb{R}^d)} \leq \|\cdot\|_{H^2_{-s}(\mathbb{R}^d)}$,
    \begin{align*}
    \|\mathfrak{Re}\{u_0 - u_\alpha\}\|_{L^2_{-s}(\mathbb{R}^d)}&\leq
    \|u_0 - u_\alpha\|_{H^2_{-s}(\mathbb{R}^d)} = \|R(z_\alpha)F - R(\omega^2)F\|_{H^2_{-s}(\mathbb{R}^d)} 
    \\
    &\leq \|R(z_\alpha) - R(\omega^2)\|\|F\|_{L^2_{s}(\mathbb{R}^d)} \overset{\alpha \rightarrow 0^+}{\longrightarrow} 0.
    \end{align*}
This implies that  $\mathfrak{Re}\{u_\alpha\} \longrightarrow\mathfrak{Re}\{u_0\}$ as $\alpha \longrightarrow 0^+$, when considered as elements of the space $L^2_{-s}(\mathbb{R}^d)$. Using our previous results we see that 
    \begin{align*}
    \|\mathfrak{Re}\{u_0\}\|_{L^2_{-s}(\mathbb{R}^d)} &= \lim_{\alpha \rightarrow 0^+} \|\mathfrak{Re}\{u_\alpha\}\|_{L^2_{-s}(\mathbb{R}^d)}
    \leq C\omega^{
    2s-2}\mathcal{N}_\omega(\beta^n)\|f\|_{L^2_{s}(\mathbb{R}^d)},
    \end{align*}
    as stated. By the same argument, the corresponding estimate holds for the $H^1_{-s}$-norm of $\mathfrak{Re}\{u_0\}$
    if assumption (H4) hold. This concludes the proof of Theorem \ref{thm:HEBound}.

\end{subsubsection}
\end{subsection}

\section{Conclusion}
We have presented the Waveholtz iteration for the unbounded domain $\mathbb{R}^d$ in the constant-coefficient case, and shown that the real part of the iterates converge as $n^{-\frac{1}{2}}$ in $H^1_{-s}$-norm to the real part of the outgoing solution to the Helmholtz equation, under suitable assumptions on the weight parameter $s$ and forcing $f$. The number of iterations required to achieve a prescribed tolerance has been shown to grow at most as $\omega^{2+}$ with the frequency parameter $\omega$, although numerical experiments suggest that the optimal growth rate is $\omega$. The key points of difference between previous analyses of the Waveholtz method is that the unbounded domain considered here required us to suitably extend the operators $\Pi$ and $\mathcal{S}$, and that the analysis could not make reference to any eigenfunction expansion of the Helmholtz solution, as is possible when considering Waveholtz on a bounded domain
with Dirichlet or Neumann boundary conditions. We instead analysed the iteration in terms of the limiting behaviour of the damped Helmholtz equation, by investigating the Fourier transform of the iterates. It remains open to prove whether the method converges for variable wave speeds.

\backmatter

\begin{appendices}

\section{Extension and boundedness of $\mathcal{S}$ and $\Pi$}\label{sec:A1}

\def\Scal{{\mathcal S}}
\def\Real{\mathbb{R}}  
\def\supp{{\rm supp}}

We aim here to prove Proposition $\ref{prop:extensionS}$
which primarily states that the operator $\mathcal{S}$, which was originally defined on the set $C^\infty_0(\mathbb{R}^d)$, can for any real $s$ and nonnegative $p$ be extended to a bounded linear operator from the weighted Sobolev space $H^p_s(\mathbb{R}^d)$ to itself.

\begin{proof}[Proof of Prop \ref{prop:extensionS}]

We start by proving \eqref{eq:SFourier0}.
This identity is true for
$u \in C^\infty_0(\mathbb{R}^d)$, which we now show in the same
way as in \cite{Waveholtz2020}. The
Fourier transform of $w$ in 
\eqref{eq:Waves} 
with $f=0$
satisfies
an ordinary differential equation for each fixed $\xi$
$$
 \partial^2_t\hat{w}+|\xi|^2\hat{w}=0.
$$
This has the solution
$\hat{w}(t,\xi)=\hat{u}(\xi)\cos(|\xi|t)$
when the initial data is $w(0,x)=u(x)$
and $w_t(0,x)=0$.
Consequently, by 
\eqref{eq:Sdef} and
\eqref{eq:betaDef},
$$
\widehat{\mathcal{S}u}(\xi) = 
\int_0^T K(t) \hat{w}(t,\xi)dt=
\hat{u}(\xi)
\int_0^T K(t) \cos(|\xi|t)dt=
\beta(\xi)\hat{u}(\xi).
$$
Finally, by the boundedness
of $\mathcal{S}$ from $L^2(\mathbb{R}^d)$
to itself, shown below,
the result extends
to all $L^2$-functions.

Next, to prove the extension of
$\mathcal{S}$ to the
weighted Sobolev spaces,
we note that since the solution operator for the wave equation is well-defined between these spaces, the operator $\mathcal{S}$ is a well-defined linear operator from $C^\infty_0(\Real^d)$ to $C^\infty_0(\Real^d)$. As the set $C^\infty_0(\mathbb{R}^d)$ is dense in $H^p_s(\mathbb{R}^d)$ for every real $s$ and non-negative $p$, the Bounded Linear Transformation Theorem implies that demonstrating 
\begin{equation}\label{Hpsestimate}
\|\Scal u\|_{H^p_s(\Real^d)}\leq C \|u\|_{H^p_s(\Real^d)},\qquad\forall
u\in C_0^\infty(\Real^d),
\end{equation}
is enough to prove the 
the claim in the
Proposition \ref{prop:extensionS}. 
Here we recall that
as defined in Section~\ref{sec:prelim},
$$
\|u\|^2_{H^p_s(\Real^d)} =
\sum_{|\alpha|\leq p} \int_{\Real^d} |\partial_x^\alpha u(x)|^2 \langle x\rangle^{2s}dx.
$$
We use the convention $L^2_s(\mathbb{R}^d)=H^0_s(\mathbb{R}^d)$ to include also $L_s^2(\mathbb{R}^d)$, and thus aim to show the bound $\eqref{Hpsestimate}$ for all $s$ and 
non-negative $p$.

For the proof we will need the open cover $\{\Omega_j\}_{j \in \mathbb{N}}$ of $\Real^d$ defined by
$$
\Omega_j =\begin{cases}
\{x\in\Real^d\ |\ |x|< T\}, & j=0,\\
\{x\in\Real^d\ |\ (j-1)T< |x|<(j+1)T\}, &j\geq 1,
\end{cases}
$$
with $T = 2\pi/\omega$, and the extended sets
$$
\Omega^e_j = \begin{cases}
  \Omega_{0}\cup\Omega_1, & j=0,\\
  \Omega_{j-1}\cup\Omega_j\cup\Omega_{j+1}, & j\geq 1.
  \end{cases}
$$
A key property of these sets is that
\begin{equation}\label{finitespeed} 
\text{if
$\supp(u) \subset \Omega_j$ then $\supp(\Scal u) \subset \Omega^e_{j}$,}
\end{equation}
by the finite speed of propagation in the wave equation. Let us begin with the case $s=0$. Take $u\in C_0^\infty(\Real^d)$. By \eqref{eq:SFourier0} and Parseval, one term in the sum for $\|\Scal u\|^2_{H^p_s(\Real^d)}$, can be estimated as
\begin{align}\label{dSterm}
\int_{\Real^d} \big|(\partial_x^\alpha \Scal u)(x)\big|^2dx&=
\int_{\Real^d} |\xi|^{2|\alpha|}|\beta(\xi)\hat{u}(\xi)|^2d\xi
\leq 
\int_{\Real^d} |\xi|^{2|\alpha|}|\hat{u}(\xi)|^2d\xi \nonumber\\
&=\int_{\Real^d} |\partial_x^\alpha u(x)|^2dx,
\end{align}
since $|\beta(\xi)|\leq 1$ for all $\xi$ by \eqref{betaEstimates}. Then \eqref{Hpsestimate} follows for $s=0$. We now consider $H^p_s(\Real^d)$ with $s\neq 0$. Let $\{\phi_j\}$ be a smooth partition of unity subordinate to $\{\Omega_j\}$, such that
$$
\sum_{j=0}^\infty \phi_j(x) = 1, \qquad \supp(\phi_j)\subset \Omega_j.
$$
To simplify the argument that follows we choose in particular $\phi_0$ given by
\[
\phi_0(x) = g\left( \frac{T-\varepsilon-|x|}{T-2\varepsilon}\right),
\]
where $0<\varepsilon<T$ is some fixed real number and $g$ is the smooth transition function defined by
\[
g(x) = \begin{cases}
    0, & x \leq 0, \\
    \frac{e^{-\frac{1}{x}}}{e^{-\frac{1}{x}} + e^{-\frac{1}{1-x}}}, &0<x<1,
    \\
    1, & 1\leq x.
\end{cases}
\]
This function $\phi_0$ is then identically zero outside of the open ball defined by $|x| < T-\varepsilon$, identically one in the closed ball defined by $|x| \in [0, \varepsilon]$, and smoothly transitioning between zero and one in the shell defined by $|x| \in (\varepsilon, T-\varepsilon)$. It can be shown that defining for each $j>0$ the function $\phi_j$ by
\[
\phi_j(x) = \phi_0\big(|x|-jT\big)
\]
produces a smooth partition of unity subordinate to the open cover given by the sets $\Omega_j$. This choice means that for any $j$ and multi-index $\alpha$, 
\begin{equation}\label{derBound}
\| \partial^\alpha \phi_j\|_{L^\infty(\mathbb{R}^d)} \leq \| \phi_0\|_{W^{|\alpha|, \infty}(\mathbb{R}^d)}.    
\end{equation}
That is, the derivatives of $\phi_j$ can be bounded independently of $j$.

Using the partition of unity $\phi_j$, we then have $\supp(u\phi_j)\subset \Omega_j$ and 
$\supp(\Scal u\phi_j)\subset \Omega_j^e$ by \eqref{finitespeed}.
Since $\Omega_j^e\cap\Omega_k^e=\emptyset$ when $|j-k|\geq 4$, it follows that $(\partial_x^\alpha\Scal u\phi_j)(\partial_x^\alpha\Scal u\phi_k)\equiv 0$ when $|j-k|\geq 4$, and consequently
\begin{align*}
|(\partial_x^\alpha \Scal u)(x)|^2&=
\Big|\sum_{j=0}^\infty (\partial_x^\alpha \Scal u\phi_j)(x)\Big|^2\leq
\Big(\sum_{j=0}^\infty | (\partial_x^\alpha \Scal u\phi_j)(x)|\Big)^2\\
&=
\sum_{j=0}^\infty\sum_{|k-j|\leq 3} \big|(\partial_x^\alpha\Scal u\phi_j)(x)\big|\big|(\partial_x^\alpha\Scal u\phi_k)(x)\big|
\\
&\leq 
\frac12\sum_{j=0}^\infty\sum_{|k-j|\leq 3} \big|(\partial_x^\alpha\Scal u\phi_j)(x)\big|^2 + \big|(\partial_x^\alpha\Scal u\phi_k)(x)\big|^2
\leq 
7\sum_{j=0}^\infty \big|(\partial_x^\alpha\Scal u\phi_j)(x)\big|^2.
\end{align*}
Hence,
\begin{align*}
\int_{\Real^d} \big|(\partial_x^\alpha\Scal u)(x)\big|^2\langle x\rangle^{2s}dx
\leq 
7\sum_{j=0}^\infty 
\int_{\Real^d} \big|(\partial_x^\alpha\Scal u\phi_j)(x)\big|^2\langle x\rangle^{2s}dx,
\end{align*}
and using \eqref{dSterm},
\begin{align*}
\lefteqn{
\int_{\Real^d} \big|(\partial_x^\alpha\Scal u\phi_j)(x)\big|^2\langle x\rangle^{2s}dx
=
\int_{\Omega_j^e} \big|(\partial_x^\alpha\Scal u\phi_j)(x)\big|^2\langle x\rangle^{2s}dx} 
\hskip 10 mm &\\
&\leq 
\sup_{x\in\Omega^e_j}\langle x\rangle^{2s}
\int_{\Omega_j^e} \big|(\partial_x^\alpha\Scal u\phi_j)(x)\big|^2dx
\leq
\sup_{x\in\Omega^e_j}\langle x\rangle^{2s}\int_{\Real^d} \big|(\partial_x^\alpha u\phi_j)(x)\big|^2dx
\\
&=
\sup_{x\in\Omega^e_j}\langle x\rangle^{2s}\int_{\Omega_j} \big|(\partial_x^\alpha u\phi_j)(x)\big|^2dx
\leq 
\frac{\sup_{x\in\Omega^e_j}\langle x\rangle^{2s}}{\inf_{x\in\Omega_j}\langle x\rangle^{2s}}\int_{\Omega_j} \big|(\partial_x^\alpha u\phi_j)(x)\big|^2\langle x\rangle^{2s}dx.
\end{align*}
Since $(1+|x|)/\sqrt{2}\leq \langle x\rangle \leq 1+|x|$,
$$
\frac{\sup_{x\in\Omega^e_j}\langle x\rangle^{2s}}{\inf_{x\in\Omega_j}\langle x\rangle^{2s}}
\leq 2^{|s|}\begin{cases}
  \left(\frac{1+(j+2)T}{1+(j-1)T}\right)^{2s} = \left(1+\frac{3T}{1+(j-1)T}\right)^{2|s|}, & s>0,\ j\geq 1\\
  \left(\frac{1+(j-2)T}{1+(j+1)T}\right)^{2s} \leq \left(1+\frac{3T}{1+(j-2)T}\right)^{2|s|}, & s<0,\ j\geq 2\\
  \left(1+2T\right)^{2s} = \left(1+2T\right)^{2|s|}, & s>0,\ j=0\\
  \left(\frac{1}{1+(j+1)T}\right)^{2s} \leq \left(1+2T\right)^{2|s|}, & s<0,\ j\leq 1.
 \end{cases}.
$$
Therefore, we obtain with $C(T,s)=2^{|s|}(1+3T)^{2|s|}$,
\begin{align*}
\int_{\Real^d} \big|(\partial_x^\alpha\Scal u)(x)\big|^2\langle x\rangle^{2s}dx
&\leq 
7C(T,s)\sum_{j=0}^\infty 
\int_{\Omega_j} \big|\big(\partial_x^\alpha( u\phi_j)\big)(x)\big|^2\langle x\rangle^{2s}dx
\\
&\leq 
7C(T, s)2^{|\alpha|}\sum_{|\gamma|\leq|\alpha|}\sum_{j=0}^\infty 
\int_{\Omega_j} |\partial_x^\gamma u(x)|^2|\partial_x^{\alpha-\gamma} \phi_j(x)|^2\langle x\rangle^{2s}dx\\
&\leq
7C(T, s)2^{|\alpha|}\|\phi_0\|_{W^{|\alpha|, \infty}(\mathbb{R}^d)}\sum_{|\gamma|\leq|\alpha|}
\int_{\Real^d} |\partial_x^\gamma u(x)|^2\langle x\rangle^{2s}dx.
\end{align*}
Note that we here needed the property \eqref{derBound} for the final step. The full norm is therefore bounded as
\begin{align*}
\|\Scal u\|^2_{H^p_s(\Real^d)} &\leq 
\sum_{|\alpha|\leq p}
\int_{\Real^d} |(\partial_x^\alpha\Scal u)(x)|^2\langle x\rangle^{2s}dx
\\
&\leq 
7C(T,s)\sum_{|\alpha|\leq p}
\|\phi_0\|_{W^{|\alpha|, \infty}}\sum_{|\gamma|\leq|\alpha|}
\int_{\Real^d} |\partial_x^\gamma u(x)|^2\langle x\rangle^{2s}dx \\
&\leq C\sum_{|\alpha|\leq p}
\int_{\Real^d} |\partial_x^\alpha u(x)|^2\langle x\rangle^{2s}dx
=C\|u\|^2_{H^p_s(\Real^d)}.
\end{align*}
This concludes the proof that $\mathcal{S}$ can be extended continuously. 

We now use this extension of $\mathcal{S}$ to show that $\Pi$ can also be extended to a continuous operator on $H^p_s(\mathbb{R}^d)$ such that,
whenever $s > \frac{1}{2}$ and $f \in L^2_s(\mathbb{R}^d)$,
\begin{equation} 
\label{affineS}
\mathcal{S}v = \mathcal{S}(v-u) + u,    
\end{equation}
where $u$ is the outgoing solution to
the Helmholtz equation \eqref{eq:HE}. Note that these requirements on $s$ and $f$ are necessary to guarantee the existence of $u$. We begin by noting that the original definition of $\Pi$ shows that for $v \in C^\infty_0(\mathbb{R}^d)$,
\[
\Pi v = \mathcal{S}v + \Pi0.
\]
This expression is then used to extend $\Pi$ to any domain that $\mathcal{S}$ has been extended to. It remains
to show that the extension satisfies \eqref{affineS}. We therefore investigate $\Pi 0$. By definition it is given by
\[
\Pi 0 = \int_0^T K(t)w(x, t)dt,
\]
where $w$ solves the initial-value problem
\begin{align}
\label{eq:waveSystem}
    \partial_t^2w  &= \Delta w -F(x, t),& \qquad (x, t) &\in \mathbb{R}^d \times (0, T), \nonumber\\
    w(x, 0) &= w_0(x),& x&\in \mathbb{R}^d,
    \\
    \partial_t w(x, 0) &= 0,& x&\in \mathbb{R}^d,
    \nonumber
\end{align}
for $F = f(x)\cos(\omega t)$ and $w_0 = 0$, which is well-posed since $f$ is assumed to be in  $L^2(\mathbb{R}^d)$. Thus, $\Pi 0$ is well-defined. To derive an expression for $\Pi 0$ we consider the limiting absorption principle and solutions 
$u_\alpha$ to the damped Helmholtz problem \eqref{eq:HEF}. One can note that for $\omega_\alpha := \sqrt{\omega^2 + i\alpha\omega}$ the function $r(x, t) = u_\alpha(x)\mathfrak{Re}\{\exp(i\omega_\alpha t)\}$ solves the problem \eqref{eq:waveSystem} with $w_0 = u_\alpha$ and $F = f(x)\mathfrak{Re}\{\exp(i\omega_\alpha t)\}$. If we finally denote by $v(x,t)$ the solution to $\eqref{eq:waveSystem}$ for $F = 0$ and $w_0 = u_\alpha$, we can combine these observations to see that, for all $\alpha>0$,
\begin{align*}
    \Pi0 &= \int_0^T K(t)wdt = \int_0^T K(t)\big(w + (r-v) - (r-v) \big)dt \\
    &=\int_0^T K(t)rdt - \mathcal{S}u_\alpha + \int_0^T K(t)(w-r+v)dt.
\end{align*}
For the final equality we used the fact that $\int_0^T K(t)vdt = \mathcal{S}u_\alpha$, by definition of $\mathcal{S}$. We now aim to see how this expression behaves as $\alpha \to 0^+$. As $\mathcal{S}$ has been extended continuously, the term $\mathcal{S}u_\alpha$ approaches $\mathcal{S}u$ as $\alpha \to 0^+$, where $u$ is the outgoing solution to the Helmholtz equation \eqref{eq:HE}. Moreover, the first term satisfies
\begin{align*}
\lim_{\alpha\to 0^+} \int_0^TK(t)rdt &= \lim_{\alpha \to 0^+} u_\alpha(x)\int_0^T K(t)\exp(-\mathfrak{Im}\{w_\alpha\}t)\cos(\mathfrak{Re}\{\omega_\alpha\}t)dt
\\
&= u(x) \int_0^T K(t)\lim_{\alpha \to 0^+}\Big(\exp(-\mathfrak{Im}\{w_\alpha\}t)\cos(\mathfrak{Re}\{\omega_\alpha\}t) \Big)dt
\\
&= u(x)\int_0^TK(t)\cos(\omega t)dt = u(x),
\end{align*}
where the exchange of the limit and the integral is justified by use of the Dominated Convergence Theorem, with the dominating function $K(t)$. If we denote by $q$ the function $w - r + v$, we see that $q$ solves the problem \eqref{eq:waveSystem} for $F(x, t) = f(x)\mathfrak{Re}\{\exp(i\omega t)-\exp(i \omega_\alpha t)\}$ and $w_0 = 0$. This function $F$ approaches zero in the $L^2$ sense, so that we by the standard estimate
\[
\|q(\cdot, t)\|_{H^1(\mathbb{R^d})} \leq C\int^t_0 \|F(\cdot, \tau)\|_{L^2(\mathbb{R}^d)}d\tau,
\]
know that for every $t \in (0, T)$, $q \to 0$ in $L^2$ as $\alpha \to 0^+$, which also means that the term $\int_0^TK(t)qdt \to 0$ in $L^2$ as $\alpha \to 0^+$. We conclude that
\[
\Pi 0 = \lim_{\alpha \to 0^+} \Big( \int_0^TK(t)rdt - \mathcal{S}u_\alpha + \int_0^T K(t)qdt \Big) = u - \mathcal{S}u,
\]
so that we can continuously extend $\Pi$ by
\[
\Pi v = \mathcal{S}v + \Pi 0 = \mathcal{S}(v-u) + u
\]
for $v \in H^p_s(\mathbb{R}^d)$.
\end{proof}

\end{appendices}

\bibliography{sn-bibliography}


\end{document}